\documentclass{conm-p-l}

\usepackage{amssymb, color}
\usepackage[dvipsnames]{xcolor}
\usepackage{mathrsfs}
\usepackage{amsmath} 
\usepackage{amsthm}
\usepackage{url} 
\usepackage{stackrel} 
\usepackage{courier}
\usepackage[all,cmtip]{xy} 
\usepackage{multicol}
\usepackage{mathtools}
\usepackage{verbatim} 
\usepackage{lipsum}
\usepackage{hyperref}
\usepackage[normalem]{ulem}
\usepackage[makeroom]{cancel}  
\usepackage{enumitem}

\catcode`~=11 
\newcommand{\urltilde}{\kern -.15em\lower .7ex\hbox{~}\kern .04em}  
\catcode`~=13 

\newcommand{\thmref}[1]{Theorem~\ref{#1}}

\newcommand{\corref}[1]{Corollary~\ref{#1}}
\newcommand{\secref}[1]{\S\ref{#1}}
\newcommand{\lemref}[1]{Lemma~\ref{#1}}
\newcommand{\eqnref}[1]{~(\ref{#1})}

\newcommand{\Inoindex}{\mathop{\mathrm{I}}\nolimits}

\newcommand{\tr}{\mathop{\mathrm{tr}}\nolimits}

\newcommand{\Aut}{\mathop{\mathrm{Aut}}\nolimits}

\makeatletter
\newif\if@borderstar
   \def\bordermatrix{\@ifnextchar*{%
       \@borderstartrue\@bordermatrix@i}{\@borderstarfalse\@bordermatrix@i*}%
   }
   \def\@bordermatrix@i*{\@ifnextchar[{\@bordermatrix@ii}{\@bordermatrix@ii[()]}}
   \def\@bordermatrix@ii[#1]#2{%
   \begingroup
     \m@th\@tempdima8.75\p@\setbox\z@\vbox{%
       \def\cr{\crcr\noalign{\kern 2\p@\global\let\cr\endline }}%
       \ialign {$##$\hfil\kern 2\p@\kern\@tempdima & \thinspace %
       \hfil $##$\hfil && \quad\hfil $##$\hfil\crcr\omit\strut %
       \hfil\crcr\noalign{\kern -\baselineskip}#2\crcr\omit %
       \strut\cr}}%
     \setbox\tw@\vbox{\unvcopy\z@\global\setbox\@ne\lastbox}%
     \setbox\tw@\hbox{\unhbox\@ne\unskip\global\setbox\@ne\lastbox}%
     \setbox\tw@\hbox{%
       $\kern\wd\@ne\kern -\@tempdima\left\@firstoftwo#1%
         \if@borderstar\kern2pt\else\kern -\wd\@ne\fi%
       \global\setbox\@ne\vbox{\box\@ne\if@borderstar\else\kern 2\p@\fi}%
       \vcenter{\if@borderstar\else\kern -\ht\@ne\fi%
         \unvbox\z@\kern-\if@borderstar2\fi\baselineskip}%
         \if@borderstar\kern-2\@tempdima\kern2\p@\else\,\fi\right\@secondoftwo#1 $%
     }\null \;\vbox{\kern\ht\@ne\box\tw@}%
   \endgroup
   }
\makeatother

\newtheorem{theorem}{Theorem}[section] 
\newtheorem{proposition}[theorem]{Proposition}
\newtheorem{lemma}[theorem]{Lemma}

\newtheorem{example}[theorem]{Example}

\newtheorem{remark}[theorem]{Remark}

\newtheorem{corollary}[theorem]{Corollary}

\allowdisplaybreaks 

\begin{document}

\title[The center of hyperelliptic Krichever-Novikov algebras]{On the module structure of the center of hyperelliptic Krichever-Novikov algebras} 
\author{Ben Cox}
\address{Department of Mathematics, College of Charleston, Charleston, SC 29424 USA}
\email{coxbl@cofc.edu} 
\thanks{The first author is partially supported by a collaboration grant from the Simons Foundation \#319261.
}

\author{Mee Seong Im}
\address{Department of Mathematical Sciences, United States Military Academy, West Point, NY 10996 USA}
\email{meeseongim@gmail.com}
\thanks{The second author acknowledges the College of Charleston for providing a productive working environment in May of 2016. 
The second author is partially supported by B. Cox's Simons Collaboration Grant \#319261 and the Department of Mathematical Sciences at the United States Military Academy.
}


  
\date{\today}

\begin{abstract}
We consider the coordinate ring  of a hyperelliptic curve and let $\mathfrak{g}\otimes R$ be the corresponding current Lie algebra where $\mathfrak g$ is a finite dimensional simple Lie algebra defined over $\mathbb C$. 
We give a generator and relations description of the universal central extension of $\mathfrak{g}\otimes R$ in terms of certain families of polynomials $P_{k,i}$ and $Q_{k,i}$ and describe how the center $\Omega_R/dR$ decomposes into a direct sum
of irreducible representations when the automorphism group is $C_{2k}$ or $D_{2k}$. 
\end{abstract}

\maketitle 

\bibliographystyle{amsalpha} 

\setcounter{tocdepth}{0} 



\section{Introduction}\label{section:intro} 
In \cite{cox-center-genus-zero-KN-algebras}, the author describes the action of the automorphism group of the ring $R=\mathbb{C}[t,(t-a_1)^{-1},\ldots, (t-a_{n})^{-1}]$ on the center of the current Krichever-Novikov algebra whose coordinate ring is $R$, where $a_1,\ldots, a_{n}$ are pairwise distinct complex numbers (see also \cite{MR3211093}).  In that setting, the five Kleinian groups $C_n$, $D_n$, $A_4$, $S_4$ and $A_5$  appear as automorphism groups of $R$ for particular choices of $a_1,\dots, a_n$. 
These five groups naturally appear in the McKay correspondence, which ties together the representation theory of finite subgroups $G$ of $SL_n(\mathbb{C})$ to the resolution of singularities of quotient orbifolds $\mathbb{C}^n/G$.

It is known that $\ell$-adic cohomology groups tend to be acted on by Galois groups, and the way in which these cohomology groups decompose can give interesting and important number theoretic
information (see for example R. Taylor's review of Tate's conjecture \cite{MR2060030}). Moreover it is an interesting and very difficult problem to describe the group $\text{Aut}(R)$ where $R$ is the space of meromorphic functions on a compact Riemann surface $X$ and to determine the module structure of its induced action on the module of holomorphic differentials $\mathcal H^1(X)$ (see \cite{MR1796706}). Now if one realizes the fact that  the cyclic homology group $HC_1(R)=\Omega^1_R/dR$ can be identified with the $H_2(\mathfrak{sl}(R),\mathbb C)$ which gives the space of $2$-cocycles (see \cite{MR618298}), it is natural to ask how $\Omega^1_R/dR$ decomposes into a direct sum of irreducible modules under the action of the $\text{Aut}(R)$.

%
One of our main results includes Theorem~\ref{mainresult}, where we describe the universal central extension of the hyperelliptic Lie algebra as a $\mathbb{Z}_2$-graded Lie algebra. 
In this theorem we give a description of the bracket of two basis elements in the universal central extension of $\mathfrak g\otimes R$ in terms of polynomials $P_{k,i}$ and $Q_{k,i}$ defined recursively
\begin{equation} 
(2k+r+3)P_{k,i}=-\sum_{j= 1 }^{r} (3j+2k-2r)a_jP_{k-r+j-1,i}
\end{equation}
for $k\geq 0$ with the initial condition
$P_{l,i}=\delta_{l,i}$, $-r \leq i,l\leq -1$ 
and 
\begin{equation} 
 ( 2m-3 )a_1Q_{m,i}  = \sum_{j=2}^{r+1 } ( 3j -2m)a_j Q_{ m - j+ 1 ,i}   
\end{equation}
with initial condition $Q_{m,i}=\delta_{m,-i}$ for $1\leq m\leq r$ and $-r\leq i\leq -1$.  In this paper $\mathfrak g$ is assumed throughout to be a finite dimensional simple Lie algebra defined over the complex numbers. 
The generating series for these polynomials can be written in terms of hyperelliptic integrals \eqnref{hyperellpiticintegral1} and \eqnref{hyperellpiticintegral2} using Bell polynomials and Fa\'a de Bruno's formula (see \secref{FaadeBruno'sFormulaandBellPolynomials}).   One can compare this result to that given in \cite{MR3487217} and also in \cite{coxzhao-2016}.

We also describe in this paper (see \thmref{theorem-result-001}) how K\"ahler differentials modulo exact forms $\Omega_R/dR$ decompose under the action of the automorphism group of the coordinate ring 
$R := R_{2}(p)=\mathbb{C}[t^{\pm 1}, u : u^2 = p(t)]$ , where $p(t)=t(t-\alpha_1)\cdots (t-\alpha_{2n})=\sum_{i=1}^{2n+1}a_it^i$, with the $\alpha_i$ being pairwise distinct roots.  In this setting, we first observe that we have the following result due to M. Bremner (see \cite{MR1303073})
\begin{equation}\label{omegaRmoddR}
 \Omega_R/dR= \bigoplus_{i=0}^{2n} \mathbb{C}\omega_i,
\end{equation}
where $\omega_0=\overline{t^{-1}\,dt}$, $\omega_{i}=\overline{t^{-i}u\,dt}$ for $i=1,\dots,2n$. 

 The possible automorphism groups for the hyperelliptic curve 
$$
R =\mathbb{C}[t^{\pm 1}, u : u^2 = t(t-\alpha_1)\cdots (t-\alpha_{2n})] 
$$
are the groups $C_{2k}$, $D_{2k}$ or one of the groups
\begin{align*}
\mathbb V_{2k}&:=\langle x,y : x^4,y^{2k},(xy)^2,(x^{-1}y)^2\rangle, \\
\mathbb U_k&:=\langle x,y: x^2,y^{2n},xyxy^{k+1} \rangle\\
Dic_k&:=\langle a,x : a^{2k}=1,\enspace x^2=a^k,\enspace x^{-1}ax=a^{-1}\rangle
\end{align*}
(see \thmref{cor:group-actions} below, \cite[ Corollary 15]{MR3631928}, \cite{MR1223022} and \cite{MR2035219}).

The above polynomials $P_{k,i}$ help us to describe how the center decomposes under the group of automorphisms of $R$.   The automorphism group of $R$ has a canonical action on $\Omega_R/dR$ and so it is natural to ask how this representation decomposes into a direct sum of irreducible representations.  When the automorphism group is $C_{2k}$ we can rewrite \eqnref{omegaRmoddR} as a direct sum of $1$-dimensional irreducible $C_{2k}$-representations. More precisely the center decomposes as: 
\begin{equation}
\Omega_R/dR\cong U_0\oplus \ldots \oplus U_{k-1}, 
 \end{equation}  
 where $\displaystyle{U_r=\bigoplus_{i\equiv r\mod k,1\leq i\leq 2n}\mathbb C\omega_i}$ for $r=1,\ldots, k-1$ is a sum of one-dimensional irreducible representation of $C_{2k}$ with character $\chi_r(s)=\exp(2\pi\imath rs/2k)$, each occurring with multiplicity $l$ and  
 $$
 U_0=\mathbb C\omega_0\oplus \bigoplus_{i=1}^l \mathbb C\omega_{ki}. 
 $$

If the automorphism group is $D_{2k}$ (with a certain parameter $c^{2n}=a_1$ and $k|2n$)
the center decomposes under the action of $D_{2k}$ as
\begin{equation}
\Omega_R/dR \cong \mathbb{C}\omega_0 
\oplus \bigoplus_{i=3}^{4} U_i^{\oplus \Upsilon_i(\epsilon_i, \nu_i)} 
\oplus 
\bigoplus_{h=1}^{k-1} V_h^{\oplus \frac{(1-(-1)^h)n}{k}}
\end{equation}
where 
\begin{align*} 
\Upsilon_i(\epsilon_i,\nu_i) 
  &= \frac{(1-(-1)^k)n}{2k}(\delta_{i,3}+\delta_{i,4})+(-1)^i\frac{1-(-1)^n}{4} +\frac{1 }{2} (-1)^i  \displaystyle{\sum_{i=n+3}^{2n}}
c^{n+3-2i}
  P_{i-n-3,-i}.
\end{align*} 
and
 $U_i$, $i=1,2,3,4$, are the irreducible one dimensional representations for $D_{2k}$ with character $\rho_i$ and $V_h$ are the irreducible 2-dimensional representations for $D_{2k}$ with character $\chi_h$, $1\leq h\leq k-1$ (see \thmref{theorem-result-001} below).  Note $\mathbb C\omega_0$ and $U_1$ are the trivial representations.

We use classical representation theory techniques found for example in  \cite{MR0450380} by Serre and \cite{MR1153249} by Fulton and Harris to prove our results. 
 
The remaining cases where the automorphism group is  $D_{2k}$ when $c^{2n}=-a_1$, $\mathbb V_{2k}$, $\mathbb U_k$ or $\text{Dic}_k$ will be studied in a future publication.

 The authors would like to thank Xiangqian Guo and Kaiming Zhao for useful discussions and pointing some corrections. \section{Background} 

\subsection{Universal Central Extensions}
An {\em extension} of a Lie algebra $\mathfrak{g}$ is a short exact sequence of Lie algebras 
\begin{equation}
\xymatrix@-1pc{
0 \ar[rr] & & \mathfrak{k} \ar[rr]^f & & \mathfrak{g}' \ar[rr]^{g} & & \mathfrak{g} \ar[rr] & & 0. 
}
\end{equation}
A homomorphism from one extension $\mathfrak{g}'\stackrel{g}{\longrightarrow} \mathfrak{g}$ 
to another extension $\mathfrak{g}'' \stackrel{g'}{\longrightarrow}\mathfrak{g}$ 
is a Lie algebra homomorphism $\mathfrak{g}'\stackrel{h}{\longrightarrow}\mathfrak{g}''$ such that 
$g'\circ h = g$. 
A central extension $\widehat{\mathfrak{g}}\stackrel{u}{\longrightarrow}\mathfrak{g}$ 
is a universal central extension if there is a unique homomorphism from 
$\widehat{\mathfrak{g}}\stackrel{u}{\longrightarrow}\mathfrak{g}$
to any other central extension $\mathfrak{g}'\stackrel{g}{\longrightarrow}\mathfrak{g}$. 

Now let $R$ be a commutative ring over $\mathbb{C}$ and let $\mathfrak{g}$ be a finite-dimensional simple Lie algebra over $\mathbb{C}$. 
Let $F=R\otimes R$ be the left $R$-module with the action $a(b\otimes c) = ab \otimes c$, where $a,b,c \in R$. 
Let $K$ be the submodule of $F$ generated by elements of the form $1\otimes ab -a\otimes b - b \otimes a$. 
Then $\Omega_R^1 = F/K$ is the module of K\"ahler differentials. 
The canonical map $d:R\rightarrow \Omega_R$ sends $d a = 1\otimes a + K$, so we will write 
$c \: da := c\otimes a + K$. 
Exact differentials consist of elements in the subspace $dR$ 
and we write $\overline{c\: da}$ as the coset of $c \: da$ modulo $dR$. 
It is a classical result by C. Kassel (1984) that the universal central extension of the current algebra $\mathfrak{g}\otimes R$ is the vector space $\widehat{\mathfrak{g}} = (\mathfrak{g}\otimes R)\oplus \Omega_R/dR$, with the Lie bracket: 
\begin{equation} 
[x\otimes a, y\otimes b] = [x,y] \otimes a b + (x,y) \overline{a\: db}, \hspace{4mm} 
[x\otimes a, \omega] = 0, \hspace{4mm} 
[\omega, \omega'] = 0, \hspace{4mm} 
\end{equation}
where $x,y\in \mathfrak{g}$, $a,b\in R$, $\omega, \omega' \in \Omega_R/dR$, and $(\cdot ,\cdot )$ is the Killing form on 
$\mathfrak{g}$. 
Since the center of the universal central extension is defined to be  $Z(\widehat{\mathfrak{g}})\subseteq \ker(\widehat{\mathfrak{g}}\rightarrow \mathfrak{g}\otimes R)$, 
Kassel showed that $Z(\widehat{\mathfrak{g}})$ is precisely $\Omega_R/dR$.  
In this paper, we will fix $R = R_2(p):=\mathbb C[t^{\pm 1}, u : u^2=p(t)]$, where  
 $p(t)=t(t-\alpha_1)\cdots (t-\alpha_{2n}) \in \mathbb{C}[t]$ and $\alpha_i$'s are pairwise distinct roots.

\subsection{Lie Algebra $2$-Cocycles}
Given a Lie algebra $\mathfrak{g}$ over $\mathbb{C}$, 
a {\em Lie algebra $2$-cocycle} for $\mathfrak{g}$ is a bilinear map $\psi:\mathfrak{g}\times \mathfrak{g}\rightarrow \mathbb{C}$ satisfying:  
\begin{enumerate}
\item $\psi(x,y) = -\psi(y,x)$ for $x,y \in \mathfrak{g}$, and 
\item $\psi([x,y],z) + \psi([y,z],x)+ \psi([z,x],y)=0$ for $x,y,z \in \mathfrak{g}$. 
\end{enumerate}
In particular, 
$\psi: (\mathfrak{g}\otimes R) \times (\mathfrak{g}\otimes R) \rightarrow \mathbb{C}$ is given by 
\begin{equation}
\psi(x\otimes a, y\otimes b) = (x,y) \overline{a \: db},
\end{equation}
which is a $2$-cocycle on $\mathfrak{g}\otimes R$.


Since we do not need the degree of the polynomial $p(t)$ to be odd, we will first let $\deg p(t)=r+1$, up until Section~\ref{section:automorphism}.  
The reason we first work in the more general setting is that it allows us to fill in the remaining case which was not covered in \cite{MR3487217} (in this manuscript, the author required that the constant term $a_0$ of $p$ to be $a_0\neq 0$).  In Sections~\ref{section:automorphism} and \ref{section:main-theorem}, we restrict to the case of $r=2n$, which allows us to use the results in \cite{MR3631928} on automorphism groups of such algebras. 
So let 
\[ 
p(t) = t(t-\alpha_1)\cdots (t-\alpha_{r})=\sum_{i=1}^{r +1 }a_i t^i, 
\] 
where the $\alpha_i$ are pairwise distinct nonzero complex numbers with $a_1=(-1)^r\prod_{i=1}^{r}\alpha_i\neq 0$ and $a_{r+1}=1$. 

Note that $R=\mathbb C[t^{\pm 1} ,u: u^2=p(t)]$ is a regular ring when $\alpha_i$ are distinct complex numbers, and $\text{Der}(R)$ is a simple infinite dimensional Lie algebra (see \cite{MR3631928}, \cite{MR829385}, \cite{MR966871} and \cite{MR2035385}).

We recall: 
\begin{lemma}[\cite{MR2813377}, Lemma 2.0.2]\label{lemma:CF11-Pki} 
If $u^m=p(t)$ and $R=\mathbb C[t^{\pm 1},u : u^m=p(t)]$, then one has in $\Omega^1_R/dR$ the congruence 
\begin{equation}\label{recursion1}
((m+1)(r+1)+im)t^{r+i}u\,dt\equiv -\sum_{j=0}^{r } ((m+1)j+mi)a_jt^{i+j-1}u\,dt\mod dR.
\end{equation}
\end{lemma}

Motivated by Lemma~\ref{lemma:CF11-Pki} with $m=2$ and $a_0=0$, we let $P_{k,i}:=P_{k,i}(a_1,\dots, a_{r})$, $k\geq -r$, 
$ -r  
\leq i\leq -1$ 
be the polynomials in the $a_i$ satisfying the recursion relations: 
\begin{equation}\label{recursion2}
(2k+r+3)P_{k,i}=-\sum_{j= 1 }^{r} (3j+2k-2r)a_jP_{k-r+j-1,i}
\end{equation}
for $k\geq 0$ with the initial condition
$P_{l,i}=\delta_{l,i}$, $-r \leq i,l\leq -1$.
\section{Cocycles}
Let $p(t)=t^{r+1}+a_{r}t^{r}+\ldots+a_1t$, where $a_i\in\mathbb C$.  
Fundamental to the description of $\widehat{\mathfrak g}$ is the following:
\begin{theorem}[\cite{MR1303073}, Theorem 3.4]   Let $R=\mathbb C[t^{\pm 1}, u :  u^2=p(t)]$.  The set
\begin{equation}\label{eqn:basis-in-tINVudt}
\{\overline{t^{-1}\,dt},\overline{t^{-1}u\,dt},\dots, \overline{t^{-r}u\,dt}\}
\end{equation}
forms a basis of $\Omega_R^1/dR$.
\end{theorem}

Let
\begin{equation}\label{eqn:basis-omega_sub_i_s}
\omega_0:=\overline{t^{-1}\,dt}\enspace \mbox{ and } \omega_k:=\overline{t^{-k}u\,dt}\quad \text{ for }1\leq k\leq r.\end{equation}
We will first describe the cocyles contributing to the {\it even} part $\mathbb C\omega_0$ of the center of the universal central extension of the hyperelliptic current algebra:

\begin{lemma}[\cite{MR1303073}, Proposition 4.2] \label{evenlem} 
 For $i,j\in\mathbb Z$ one has
\begin{equation}
\overline{t^i\,d(t^j)}= j \delta_{i+j,0}\omega_0
\end{equation}
and 
\begin{equation}
\overline{t^{i}u\,d(t^{j}u)}=\sum_{k=1}^{r+1}\left(j+\frac{1}{2}k\right)a_k\delta_{i+j,-k}\omega_0.
\end{equation}
\end{lemma}

%

For the {\it odd} part $\mathbb C\omega_{1}\oplus \ldots \oplus \mathbb C\omega_{r}$ of the center, we generalize Proposition 4.2 in \cite{MR1303073} via the following result:  
\begin{proposition}\label{oddlem}  
 For $i,j\in\mathbb Z$, one has
\begin{equation}\label{eqn:omega_i-Pij-Qij}
\overline{t^{i}u\,d(t^{j})}=j\begin{cases} 
	\displaystyle{\sum_{k=1}^{r}} P_{i+j-1,-k}  \omega_{k} &\quad \text{ if } i + j \geq -r+1, \\
 	\displaystyle{\sum_{k=1}^{r}}  Q_{-i-j+1,-k} \omega_{k} 			&\quad \text{ if } i + j <  -r+1\color{black},  \\ 
\end{cases}
\end{equation} 
where $P_{m,i}$ is the recursion relation in Equation~\eqnref{recursion2}  
and $Q_{m,i}$ satisfies 
\begin{equation}\label{eq:recursion-relation-for-Qmi}
 ( 2m-3 )a_1Q_{m,i}  = \left(\sum_{j=2}^{r+1 } ( 3j -2m)a_j Q_{ m - j+ 1 ,i}  \right)
\end{equation}
with initial condition $Q_{m,i}=\delta_{m,-i}$ for $1\leq m\leq r$ and $-r\leq i\leq -1$.
\end{proposition}

\begin{proof}
We set $m=2$ and replace $j$ in the summation in Equation~\eqref{recursion1} by $k$, and then replace $i$ with $-r+i+j-1$
to obtain:   
%
\[  
(2(i+j)+ r+1)t^{i+j-1}u\,dt\equiv -\sum_{k=1}^{r } (3k+2(i+j)-2(r+1))a_kt^{i+j-(r+1)+k-1}u\,dt\mod dR,   
\color{black} 
\]
and similarly
\begin{equation}\label{eqn:P-ij-only-recursion-equation}
(2(i+j)+ r+1)P_{i+j-1,\iota} =  -\sum_{k=1}^{r } (3k+2(i+j)-2(r+1))a_k P_{i+j-1 +k -(r +1) , \iota}.  	\hspace{4mm}  \color{black} 
\end{equation}
%
%

%
%

So now assume for $\iota \geq -r$,  $\hspace{4mm}$  
\color{black}

\begin{equation}\label{eqn:base-step-induction-dt-P-ij}
\overline{t^{\iota } u\,dt}
=\sum_{k=1}^{r} P_{ {\iota }, k - (r+1)} \omega_{r + 1 - k}.
\end{equation}
It is clear that Equation~\eqref{eqn:base-step-induction-dt-P-ij} holds when ${\iota } = - r , \dots, -1$ as $P_{l,i}=\delta_{l,i}$ for $-r\leq i,l\leq -1$.   
Then the induction step is: 
\begin{align*} 
\overline{t^{\iota +1}u\,dt} &=  -\sum_{k=1}^{r }\left( \dfrac{ 3k + 2\iota - 2r + 2 }{ 2 \iota + r + 5 }\right) a_k \overline{t^{\iota +k - r } u \: dt} \\   
&=  - \sum_{l=1}^{r} \sum_{k=1}^{r }\left( \dfrac{ 3k + 2\iota - 2r + 2 }{ 2 \iota  + r + 5 }\right) a_k   P_{ {\iota +k  - r }, l -( r+1)} \omega_{r + 1 - l} \\ 
&=\sum_{l=1}^{r}P_{\iota + 1 ,  l -(r  + 1)}  \omega_{r+1-l}.    
\end{align*}
\color{black} 
Now, for $i+j\geq -r +1 $, we have 
\begin{equation}
\overline{t^{i}u\,d(t^j)} 
= j\overline{t^{i+j-1}u\,dt} 
= j \sum_{l=1}^{r}P_{i+ j -1 ,  l -(r  + 1)}  \omega_{r+1-l} 
= j \sum_{k=1}^{r} P_{i+ j -1 ,  -k } \omega_k.
\end{equation}
\color{black}

Again consider  \eqnref{recursion1} and set $r+i = k-1$ or $i = k - (r + 1)$: 
\begin{equation}
( 2k + r + 1   )  \overline{ t^{k-1} u\: dt}  = - \sum_{j=1}^{r } ( 3j+2k - 2(r + 1)  )a_j \overline{ t^{k +j -1 - (r + 1)}u\,dt},  
\end{equation}
 and write it as
\begin{equation}\label{eq:Q-ij-but-using-tj-u-dt}					
(   -2(m -1) + r + 1  )  \overline{ t^{-m} u\: dt}  = - \sum_{j=1}^{r } ( 3j -2m +2  - 2(r + 1)  )a_j \overline{ t^{-( m - j + r + 1 )}u\,dt}. 
\end{equation}
Then 
\begin{align*}
0  &= - \sum_{j=1}^{r +1} ( 3j+2k - 2(r + 1)  )a_j \overline{ t^{k +j -1 - (r + 1)}u\,dt}  \\
&= - ( 2k - 2r + 1  )a_1 \overline{ t^{k  - (r + 1)}u\,dt}-\ldots - ( 2k +r -2)  )a_r \overline{ t^{k-2 }u\,dt} -  (2k +r + 1  )\overline{ t^{k-1 }u\,dt},
\end{align*}
as $a_{r+1}=1$.   We rewrite this as 
\begin{align*}
 \overline{ t^{k  - (r + 1)}u\,dt}  &=\frac{-1}{( 2k - 2r + 1  )a_1}\left( ( 2k - 2r+4  )a_2 \overline{ t^{k -r}u\,dt} +\ldots \right. \\
 &\hspace{4mm}\left. +( 2k +r -2)  a_r \overline{ t^{k-2 }u\,dt} +  (2k +r + 1  )\overline{ t^{k-1 }u\,dt}\right) \\
 &=\frac{-1}{( 2k - 2r + 1  )a_1}\left(\sum_{j=2}^{r +1} ( 3j+2k - 2(r + 1)  )a_j \overline{ t^{k +j -1 - (r + 1)}u\,dt}\right).
\end{align*}
For $k=0,-1,-2$ we have for instance
\begin{align*}
\overline{ t^{ - (r + 1)}u\,dt}  &=\frac{1}{(  - 2r + 1  ) a_1}\left( - ( - 2r+4  )a_2 \overline{ t^{-r}u\,dt} -\ldots - ( r -2)  a_r \overline{ t^{-2 }u\,dt} -  (r + 1  )\overline{ t^{-1 }u\,dt}\right) \\
 \overline{ t^{  - r -2}u\,dt}  &= \frac{1}{( - 2r -1  )a_1}\left(- ( - 2r+2  )a_2 \overline{ t^{-r-1}u\,dt} -\ldots - ( r -4)  a_r \overline{ t^{-3 }u\,dt} -  (r -1  )\overline{ t^{-2 }u\,dt} \right) \\
 \overline{ t^{ - r -3}u\,dt}  &=\frac{1}{(  - 2r -3  )a_1}\left( - (  - 2r  )a_2 \overline{ t^{ -r-2}u\,dt} -\ldots - ( r -6)  a_r \overline{ t^{-4 }u\,dt} -  (r -3  )\overline{ t^{-3 }u\,dt}\right) .
\end{align*}
Setting $-m=k-r-1$, we get $k=-m+r+1$ and 
\begin{align}
 \overline{ t^{-m}u\,dt}   &=\frac{1}{( 2m-3 )a_1}\left(\sum_{j=2}^{r +1} ( 3j-2m )a_j \overline{ t^{-m+j -1 }u\,dt}\right) \label{trecursion}
\end{align}
for $m\geq r+1$. This leads us to the recursion relation: 
\begin{equation}\label{recursion3.2-edited} 
 Q_{m,i}  = \frac{1}{( 2m-3 )a_1}\left(\sum_{j=2}^{r+1 } ( 3j -2m)a_j Q_{ m - j+ 1 ,i}  \right)
\end{equation}
\color{black}
for $m\geq r+1$ with the initial condition
$Q_{m,i}=\delta_{m,-i}$, $1\leq m\leq r$ and $-r\leq i\leq -1$. 
\color{black}

So now assume for $\iota \geq 1$,  
\begin{equation}\label{recursioneq:Q-ij}
\overline{t^{-\iota }u\,dt}
= \sum_{k=0}^{r-1}  Q_{\iota, k-r}\omega_{r - k}= \sum_{k=1}^{r}  Q_{\iota, -k}\omega_{k}.
\end{equation}
It is clear that Equation~\eqref{recursioneq:Q-ij} holds for $\iota =1,\dots ,r$ 
as $Q_{m,i}=\delta_{m,-i}$, $1\leq m\leq r$ and $-r\leq i\leq -1$.

For $\iota \geq r$, we have by \eqnref{trecursion}, \eqnref{recursion3.2-edited} and the induction hypothesis: 
\color{black}


 \begin{align*}
\overline{t^{-(\iota +1)}u\,dt} &= \sum_{j=2}^{r +1} \frac{( 3j-2\iota -2)a_j}{( 2\iota-1 )a_1} \overline{ t^{-\iota+j -2 }u\,dt}\\
&= \sum_{k=0}^{r-1}\sum_{j=2}^{r+1}\frac{( 3j-2i -2)a_j}{( 2i-1 )a_1} Q_{\iota  - j+2, k - r}\omega_{r-k}  \\
&=   \sum_{k=0}^{r-1}Q_{\iota + 1, k - r}\omega_{r-k}, 
\end{align*}
which proves \eqnref{trecursion} for $m=\iota+1$. 

We conclude for $i+j-1< -r $, we have 
\begin{equation}
\overline{t^{i}u\,d(t^j)}=j\overline{t^{i+j-1}u\,dt}= j \sum_{k=1}^{r}  Q_{ -i - j  + 1 , k-(r+1)}\omega_{r+1 - k} 
= j \sum_{k=1}^{r}  Q_{ -i - j  + 1 , -k }\omega_{ k}. 
\end{equation}

\color{black} 
\end{proof}

\section{ Fa\'a de Bruno's Formula and Bell Polynomials}\label{FaadeBruno'sFormulaandBellPolynomials}
 Now consider the formal power series 
 
\begin{equation}
P_i(z):=P_i(a_1,\dots, a_{r}, z):=\sum_{k\geq -r} P_{k,i} z^{k+r} = \sum_{k\geq 0} P_{k-r,i} z^k 
\end{equation}
for $-r\leq i\leq -1$. 
We will find an integral formula for $P_i(z)$ below. 
One can show that $P_i(z)$ must satisfy the first order differential equation
\begin{equation}
\frac{d}{dz}P_i(z)-\frac{Q(z)}{2zT(z)}P_i(z)=\frac{R_i(z)}{2zT(z)}, 
\end{equation}
\color{black}
where
\begin{gather*}
T(z):=\sum_{j=1}^{r+1} a_j z^{r+1-j},\quad Q(z) := z T'(z) +(r-3) T(z),  
\end{gather*}
and
\begin{gather*}
R_i(z):=\sum_{j=1}^{r+1} \left(\sum_{ 1-j \leq k < 0 }( 3j + 2k - 2r ) a_j \delta_{k + j -r-1 , i} z^{k+r} \right)   
\end{gather*}
since indeed, we have  
\begin{align*}
2zT(z)\frac{d}{dz}P_i(z)-Q(z)P_i(z)
&=  \sum_{k\geq 0}\left(  \sum_{j=1}^{r+1} 2ka_j P_{k-r,i}z^{r+k+1-j}   - \sum_{j=1}^{r+1} (2r-j-2)a_j  P_{k-r,i}z^{r+k+1-j} \right)  \\
&=   \sum_{k\geq 0} \left( \sum_{j=1}^{r+1} (2k-2r+j+2)a_jP_{k-r,i} z^{r+1+k-j}\right) \\
&=   \sum_{k\geq 0} \left(  \sum_{j=1}^{r+1} (2k+3j-2r) a_j P_{k+j-r-1, i } z^{r+k}\right) \\
&\quad + \sum_{j=1}^{r+1}\left(  \sum_{1-j\leq k< 0} (2k+3j-2r) a_j P_{k+j-r-1,i} z^{r+k}\right) \\
&=R_i(z), 
\end{align*}
where the first summation in the second to last equality is zero due to \eqnref{recursion2}.
\color{black}

An integrating factor is 
$$
\mu(z)=\exp\int - \frac{Q(z)}{2zT(z)} \,dz
=\frac{1}{z^{(r-3)/2}\sqrt{T(z)}}, 
$$
and so 
\begin{equation}\label{hyperellpiticintegral1}
P_i(z):=z^{(r-3)/2}\sqrt{T(z)}\int \frac{R_i(z)}{2z^{(r-1)/2}T(z)^{3/2}} \,dz.
\end{equation}
The way we interpret the right hand hyperelliptic integral ($T(0)$ $=$ $a_{r+1}$ $=$ $1\neq 0$) is to expand $R_i(z)/T(z)^{3/2}$ in terms of a Taylor series about $z=0$ and then formally integrate term by term. We then multiply the result by series for 
$z^{(r-3)/2}\sqrt{T(z)}$. 
 Let us explain this more precisely.

One can expand both $\sqrt{T(z)}$ and $1/T(z)^{3/2}$ using Bell polynomials and Fa\`a di Bruno's formula as follows.  
Bell polynomials in the variables $z_1,z_2,z_3,\ldots, z_{m-k+1}$ are defined to be 
\begin{align*}
B_{m,k}(z_1,\dots,z_{m-k+1}):=\sum \frac{m!}{l_1!l_2!\cdots l_{m-k+1}!}\left(\frac{z_1}{1!}\right)^{l_1} \cdots \left(\frac{z_{m-k+1}}{(m-k+1)!!}\right)^{l_{m-k+1}}, 
\end{align*}
where the sum is over $l_1+l_2+\ldots +l_{m-k+1}=k$ and $l_1+2l_2+3l_3+\ldots +(m-k+1)l_{m-k+1}=m$ (see \cite{MR1502817}).
\color{black}
 
Now Fa\`a di Bruno's formula  (\cite{dB1} and \cite{dB2}; discovered earlier by Arbogast \cite{Arbogast}) for the $m$-derivative of $f(g(x))$ is 
\begin{align*}
\frac{d^m}{dx^m}f(g(x))=\sum_{l=0}^mf^{(l)}(g(x))B_{m,l}(g'(x),g''(x),\dots, g^{(m-l+1)}(x)).
\end{align*}
Here $f(x)=x^{-3/2}$, $g(x)=T(x)$, so we get 
\begin{equation}
f^{(m)}(x)=\frac{(-1)^m(2m+1)!!}{2^mx^{(2m+3)/2}}
\end{equation}
where 
$$
(2k-1)!!=\Gamma(k+(1/2))2^k/\sqrt{\pi}.
$$  
Then  $(-1)!! =1$
and $T^{(k)}(0)=k!a_{r+1-k}$ so that 
\begin{align*}
\left. \frac{d^m}{dx^m}f(g(x))\right\vert_{x=0}=\sum_{l=0}^m\frac{(-1)^ l(2 l+1)!!}{2^ l }B_{m, l}(a_{r},2a_{r-1},\dots,(m- l+1)!a_{r-m+l}).
\end{align*}
As a consequence 
\begin{align*}
\frac{1}{T(z)^{3/2} }&=\sum_{m=0}^\infty \frac{1}{m!}\left. \frac{d^m}{d z^m}f(g( z))\right\vert_{ z=0} z^m \\
&=\sum_{m=0}^\infty \frac{1}{m!} \left(\sum_{l=0}^m\frac{(-1)^ l(2 l+1)!!}{2^ l }B_{m, l}(a_{m-1},2a_{m-2},\dots,(m- l+1)!a_{r-m+l})\right) z^m ,
\end{align*}
and hence
\begin{equation}
T_m(a_1,\dots,a_{m-1})= \frac{1}{m!}  \sum_{l=0}^m\frac{(-1)^ l(2 l+1)!!}{2^ l }B_{m, l}(a_{m-1},2a_{m-2},\dots,(m- l+1)!a_{r-m+l}),
\end{equation}
where $T_m(a_1,\dots,a_{m-1})$ are defined through the equation
\begin{align*}
\frac{1}{T(z)^{3/2} }&=\sum_{m=0}^\infty T_m(a_1,\dots,a_{m-1})z^m.
\end{align*}

Similarly for $\sqrt{T(z)}$, we set $f(z)=\sqrt{z}$ so that 
$$
f^{(k)}(z)=\frac{(-1)^{k+1}(2k-3)!!}{2^kz^{(2k-1)/2}}
$$ 
for $m\geq 0$ 
and thus 
\begin{align*}
\sqrt{T(z)}&=\sum_{m=0}^\infty \frac{1}{m!} \left(\sum_{l=0}^m\frac{(-1)^{l+1}(2l-3)!!}{2^l }B_{m, l}(a_{m-1},2a_{m-2},\dots,(m- l+1)!a_{l-1})\right) z^m.
\end{align*}

We then form the formal power series 
\begin{equation}
Q_i(z):=Q_i( a_1,\dots, a_{r}, z) = \sum_{k\geq r+2} Q_{k-(r+1), i} z^k = \sum_{k\geq 1}Q_{k,i}z^{k+r+1}
\end{equation} 
for $1\leq i\leq r+1$.  
Similar to above, we see that this formal series must satisfy 
\begin{equation}  
\frac{d}{dz}Q_i(z)-\frac{Q(z)}{2z P(z)}Q_i(z)=\frac{S_i(z)}{2z P(z)},   
\end{equation}  
where 
\begin{gather}  
P(z):=\sum_{j=1}^{r+1} a_j z^{j},  \quad   Q(z) := z P'(z) + 2 (r+2) P(z) , \label{pqs} 
\end{gather}  
and
\begin{gather*}
   S_i(z):=-\sum_{m=1}^{r+1} \left(  \sum_{j=1}^{m-1} (3j -2m+2)a_jQ_{m-j,i}\right)z^{ m + r + 1 }.    \end{gather*}
Indeed
\begin{align*}
2z  P(z)\frac{d}{dz}Q_i(z)-Q(z)Q_i(z)
&=   \sum_{k\geq 1}\sum_{j=1}^{r+1} (2(k+r+1) - j -2(r+1) -2)a_jQ_{k,i}z^{j+k+r+1}\\
&= -   \sum_{k\geq 1}\sum_{j=1}^{r+1} (j-2k+2) a_j Q_{k,i} z^{j+k+r+1}\\
&=  - \sum_{m\geq r+2} \left( \sum_{j=1}^{r+1} (3j -2m+2) a_j Q_{m-j,i} \right) z^{m+r+1} \\
&\quad - \sum_{m=1}^{r+1} \left(  \sum_{j=1}^{m-1} (3j -2m+2)a_jQ_{m-j,i}\right)z^{m+r+1} \\
&=S_i(z).
\end{align*}

 An integrating factor is 
$$
\mu(z)=\exp\int - \frac{Q(z)}{2z  P(z)} \,dz
=\frac{1}{z^{r+2}\sqrt{ P(z)}},
$$
and so 
\begin{equation}\label{hyperellpiticintegral2}
Q_i(z):=  z^{r+2}\sqrt{  P(z)} \int \frac{S_i(z)}{2 z^{r+3} P(z)^{3/2}} \,dz.
\end{equation}
\color{black}

\subsection{Example}    Let $p(t)=t^{2n+1}-t=t(t-\zeta)(t-\zeta^2)\cdots (t-\zeta^{2n})$, where $\zeta=\exp(\pi\imath/n)$ is a primitive $2n$-th root of unity.  Then $a_1=-1$ and $a_j=0$ for $2\leq j\leq 2n$, and hence
the recursion relation \eqnref{recursion2} becomes 
\begin{equation}
P_{k,i}=\frac{-4n+2 k +3}{2n+2k+3}P_{k-2n,i}, 
\end{equation}
where $k\geq 0$. Since $P_{\ell,i}=\delta_{\ell,i}$ for all $-2n \leq \ell, i\leq -1$, 
the closed form is: 
\begin{equation}
P_{k,i} = \prod_{j=1}^{s} \dfrac{-4(jn)+2k + 3}{(2-4(j-1))n+2k + 3}P_{k-s(2n),i} \hspace{2mm} \mbox{ where } k\geq 0. 
\end{equation}
This implies when $k=s(2n)+i$ or $s=\frac{k-i}{2n}$, we have  
\[ 
P_{k,i} = \prod_{j=1}^{\frac{k-i}{2n}} \dfrac{-4(jn)+2k + 3}{(2-4(j-1))n+2k + 3} \hspace{2mm}\mbox{ where } k\geq 0. 
\]

Similarly, the recursion relation \eqref{eq:recursion-relation-for-Qmi} becomes 
\begin{equation}
Q_{k,i} = \dfrac{6n-2k+3}{3-2k}Q_{k-2n,i},  
\end{equation}
with initial conditions $Q_{m,i}=\delta_{m,-i}$, where $1\leq m\leq 2n$ and $-2n\leq i\leq -1$. 
So the closed form is
\begin{equation}
Q_{k,i} =  \prod_{j=1}^{v} \dfrac{(6+4(j-1))n-2k + 3}{(4(j-1))n-2k + 3}Q_{k-v(2n),i} \hspace{2mm}\mbox{ where } k\geq 0. 
\end{equation} 
So when $k=v(2n)-i$ or $v=\frac{k+i}{2n}$,  then 
\[ 
Q_{k,i} =  \prod_{j=1}^{\frac{k+i}{2n}} \dfrac{(6+4(j-1))n-2k + 3}{(4(j-1))n-2k + 3}.
\] 

Thus, 
\begin{equation}
P_i(z) = P_i(-1,0,\ldots, 0,z) = \sum_{k\geq 0} \prod_{j=1}^{\frac{k-i}{2n}-1} \dfrac{-4(j+1)n+2k+3}{(2-4j)n+2k+3}\delta_{\bar k,\bar i}  z^k, 
\end{equation}
where $-2n \leq i\leq -1$, $\bar a$ is the congruence class of $a \! \!\mod 2n$,  
and 
\begin{equation}
Q_i(z) = Q_i(-1,0,\ldots, 0,z) = \sum_{k\geq 1}  \prod_{j=1}^{\frac{k+i}{2n}} \dfrac{(6+4(j-1))n-2k + 3}{(4(j-1))n-2k + 3} \delta_{\overline {k+i},\bar 0} z^{k+2n+1},  
\end{equation}
where $1\leq i\leq 2n+1$.

\subsection{Example}  We consider now the particular example $
p(t)=t^5-2c t^3+t $.   Here $r=4$, $a_0=0=a_2=a_4$, $a_1=1=a_5$ and $a_3=-2c$.    We have 
\begin{gather*}
T(z)=z^4-2cz^2+1,
\end{gather*}
and
\begin{align*}
R_{-1}(z)&=\sum_{j=1}^{5} \left(\sum_{ 1-j \leq k < 0 }( 3j + 2k - 8 ) a_j \delta_{k + j, 4} z^{k+4} \right) =5z^3,\\
R_{-2}(z)&=\sum_{j=1}^{5} \left(\sum_{ 1-j \leq k < 0 }( 3j + 2k - 8 ) a_j \delta_{k + j , 3} z^{k+4} \right)=3z^2 ,\\
R_{-3}(z)&=\sum_{j=1}^{5} \left(\sum_{ 1-j \leq k < 0 }( 3j + 2k - 8 ) a_j \delta_{k + j  , 2} z^{k+4} \right)=2cz^3+z, \\
R_{-4}(z)&=\sum_{j=1}^{5} \left(\sum_{ 1-j \leq k < 0 }( 3j + 2k - 8 ) a_j \delta_{k + j , 1} z^{k+4} \right)=6cz^2-1.
\end{align*}
Thus 
\begin{align*}
P_{-1}(z)&=5z^{1/2}\sqrt{z^4-2cz^2+1}\int \frac{z^3}{2z^{3/2}(z^4-2cz^2+1)^{3/2}} \,dz \\
&= z^3+\frac{2 c z^5}{3}+\left(\frac{28 c^2}{39}-\frac{1}{13}\right) z^7+\left(\frac{616 c^3}{663}-\frac{196 c}{663}\right)
   z^9 \\
   &\hspace{4mm}+\frac{\left(6160 c^4-3388 c^2+153\right) z^{11}}{4641}+O\left(z^{13}\right),\\  \\
P_{-2}(z)&=3z^{1/2}\sqrt{z^4-2cz^2+1}\int \frac{z^2}{2z^{3/2}(z^4-2cz^2+1)^{3/2}} \,dz \\
&=z^2+\frac{2 c z^4}{7}+\left(\frac{20 c^2}{77}+\frac{1}{11}\right) z^6+\left(\frac{24 c^3}{77}+\frac{4 c}{77}\right)
   z^8 \\ 
   &\hspace{4mm}+\left(\frac{624 c^4}{1463}-\frac{36 c^2}{1463}-\frac{7}{209}\right) z^{10} +O\left(z^{12}\right),  \\  \\
P_{-3}(z)&=z^{1/2}\sqrt{z^4-2cz^2+1}\int \frac{2cz^3+z}{2z^{3/2}(z^4-2cz^2+1)^{3/2}} \,dz \\
&= z+\frac{z^5}{3}+\frac{14 c z^7}{39}+\left(\frac{308 c^2}{663}-\frac{5}{51}\right) z^9+\left(\frac{440
   c^3}{663}-\frac{1364 c}{4641}\right) z^{11}+O\left(z^{13}\right), \\  \\
P_{-4}(z)&=z^{1/2}\sqrt{z^4-2cz^2+1}\int \frac{6cz^2-1}{2z^{3/2}(z^4-2cz^2+1)^{3/2}} \,dz \\
&= 1+\frac{5 z^4}{7}+\frac{50 c z^6}{77}+\left(\frac{60 c^2}{77}-\frac{1}{7}\right) z^8+\frac{12 c \left(130 c^2-53\right)
   z^{10}}{1463}+O\left(z^{13}\right).  
\end{align*}
 Here the integrals are from $0$ to $z$.  
 
The polynomials 
$P_{k,i} = P_{k,i}(c)$ satisfy the recursion:
\begin{equation} 
(2k+7)P_{k,i}=-(2k-5)P_{k-4,i}+2c(2k+1)P_{k-2,i}
\end{equation}
for $k\geq 0$ with initial conditions
$P_{l,i}=\delta_{l,i}$, $-r \leq i,l\leq -1$.
\color{black}
We see that $P_{k,i}$ agree with the coefficients given above in the generating series. 

To get explicit generating formulae for the $Q_{k,i}$ (see \eqnref{recursion3.2-edited}), we have 
\begin{align*}
   S_{-1}(z)&=-\sum_{m=1}^{5} \left(  \sum_{j=1}^{m-1} (3j -2m+2)a_jQ_{m-j,-1}\right)z^{ m +5 }= -z^7+6cz^9, \\
     S_{-2}(z)&=-\sum_{m=1}^{5} \left(  \sum_{j=1}^{m-1} (3j -2m+2)a_jQ_{m-j,-2}\right)z^{ m + 5 }=z^8+2cz^{10}, \\
        S_{-3}(z)&=-\sum_{m=1}^{5} \left(  \sum_{j=1}^{m-1} (3j -2m+2)a_jQ_{m-j,-3}\right)z^{ m + 5 }=3z^9, \\
           S_{-4}(z)&=-\sum_{m=1}^{5} \left(  \sum_{j=1}^{m-1} (3j -2m+2)a_jQ_{m-j,-4}\right)z^{ m +5 }=5z^{10}, 
 \end{align*}
and thus 
\begin{align*}
Q_{-1}(z)&= z^{6}\sqrt{ z^5-2c z^3+z} \int \frac{-z^7+6cz^9}{2 z^{7} (z^5-2c z^3+z)^{3/2}} \,dz\\
&= z^6+\frac{5 z^{10}}{7}+\frac{50 c z^{12}}{77}+\left(\frac{60 c^2}{77}-\frac{1}{7}\right) z^{14}+\frac{12 c \left(130
   c^2-53\right) z^{16}}{1463}+O\left(z^{18}\right) ,\\  \\
Q_{-2}(z)&=z^{6}\sqrt{  z^5-2c z^3+z} \int \frac{z^8+2cz^{10}}{2 z^{7} (z^5-2c z^3+z)^{3/2}} \,dz\\
&= z^7+\frac{z^{11}}{3}+\frac{14 c z^{13}}{39}+\frac{1}{663} \left(308 c^2-65\right) z^{15}+\frac{44 c \left(70 c^2-31\right)
   z^{17}}{4641}+O\left(z^{19}\right),  \\  \\
Q_{-3}(z)&=z^{6}\sqrt{ z^5-2c z^3+z} \int \frac{3z^9}{2 z^{7} (z^5-2c z^3+z)^{3/2}} \,dz \\
&=z^8+\frac{2 c z^{10}}{7}+\frac{1}{77} \left(20 c^2+7\right) z^{12}+\frac{4}{77} \left(6 c^3+c\right)
   z^{14} \\
   &\hspace{4mm}+\frac{\left(624 c^4-36 c^2-49\right) z^{16}}{1463}+O\left(z^{17}\right) , \\  \\
Q_{-4}(z)&=z^{6}\sqrt{  z^5-2c z^3+z} \int \frac{5z^{10}}{2 z^{7} (z^5-2c z^3+z)^{3/2}} \,dz\\
&=z^9+\frac{2 c z^{11}}{3}+\frac{1}{39} \left(28 c^2-3\right) z^{13}+\frac{28}{663} c \left(22 c^2-7\right)
   z^{15} \\
   &\hspace{4mm}+\frac{\left(6160 c^4-3388 c^2+153\right) z^{17}}{4641}+O\left(z^{19}\right).  
\end{align*} 
The recurrence relation for $Q_{m,i}$ is \eqnref{recursion3.2-edited}: 
\begin{equation}
 Q_{m,i}  = \frac{1}{( 2m-3 )a_1}\left(\sum_{j=2}^{5 } ( 3j -2m)a_j Q_{ m - j+ 1 ,i}\right) = \frac{2c( 2m-9) Q_{ m -2 ,i}+(15-2m) Q_{ m -4 ,i}}{2m-3 } 
\end{equation}
for $m\geq 5$ and $Q_{m,i}=\delta_{m,-i}$, $1\leq m\leq 4$.  This agrees with the coefficients of the generating series given above for $Q_i(z)$. 

\subsection{Example} Let us take $p(t)=t^7-2bt^4+t$.     For this example, we limit ourselves to writing down just the first few terms of the generating series $P_{-1}(z)$.  The recursion relation for the $P_{k,i}$'s using \eqnref{recursion2} is 
\begin{equation}
(2k+9)P_{k,i}=-\sum_{j= 1 }^{6} (3j+2k-12)a_jP_{k-r+j-1,i}=4bk P_{k-3,i}-(2k-9)P_{k-6,i}
\end{equation}
for $k\geq 0$ with the initial condition
$P_{l,i}=\delta_{l,i}$, $-6\leq i,l\leq -1$. One can calculate by hand for example the first three nonzero nonconstant polynomials for $i=-1$, which are 
\begin{align*} P_{2,-1}=\frac{8 b}{13}, \quad P_{5,-1}=\frac{(-13 + 160 b^2)}{247},\quad  P_{8,-1}\frac{8 b (-37 + 
   128 b^2)}{1235}.
   \end{align*}
In this setting of $p(t)$, we have
$ R_{-1}(z)=7z^5$, $T(z)=z^6-2bz^3+1$ 
and as an example using Fa\`a di Bruno's formula and Bell polynomials, we get
\begin{align*}
P_{-1}(z)&=z^{3/2}\sqrt{z^6-2bz^3+1}\int \frac{7z^5}{2z^{5/2}(z^6-2bz^3+1)^{3/2}} \,dz \\
&=z^5+\frac{8 b z^8}{13}+\frac{1}{247} \left(160 b^2-13\right) z^{11}+\frac{8 b \left(128 b^2-37\right)
   z^{14}}{1235}+O\left(z^{17}\right).
\end{align*}
Note in the integral we take the constant of integration to be $0$.

\section{Lie algebra generators and relations for $\widehat{\mathfrak g\otimes R}$.}

Theorem~\ref{mainresult} is a generalization of the main theorem in  \cite{MR2373448}. 

\begin{theorem}\label{mainresult}  Let $a_1\neq 0$.   Let $\mathfrak g$ be a simple finite dimensional Lie algebra over the complex numbers with Killing form $(\,\cdot\,|\,\cdot \,)$ and for ${\bf a}=(a_1,\dots, a_{r})$ define $\psi_{ij}(\mathbf{a})\in\Omega_R^1/dR$ by
\begin{equation}
\psi_{ij}(\bf a)= 
\begin{cases} 
\displaystyle{\sum_{k=1}^{r}}  P_{i+j-1,-k} \omega_{k} &\quad \text{ if } i + j \geq -r +1,\\
\displaystyle{ \sum_{k=1}^{r}} Q_{-i-j+1,-k}\omega_{k} &\quad \text{ if } i + j < - r + 1.
 \\
\end{cases}
\end{equation} 
The universal central extension of the hyperelliptic Lie algebra $\mathfrak g \otimes R$ is the $\mathbb Z_2$-graded Lie algebra 
$$
\widehat{\mathfrak g}=\widehat{\mathfrak g}^0\oplus \widehat{\mathfrak g}^1,
$$
where
$$
\widehat{\mathfrak g}^0=\left(\mathfrak g\otimes \mathbb C[t,t^{-1}]\right)\oplus \mathbb C\omega_{0},\qquad \widehat{\mathfrak g}^1=\left(\mathfrak g\otimes \mathbb C[t,t^{-1}]u\right)\oplus\bigoplus_{k=1}^{r}\left( \mathbb C\omega_{k}\right)$$
with bracket
\begin{align}
[x\otimes t^i,y\otimes t^j]&=[x,y]\otimes t^{i+j}+\delta_{i+j,0}j(x,y)\omega_0, \label{even1}\\ \notag \\
[x\otimes t^{i}u,y\otimes t^{j}u]&=[x,y]\otimes t^{i+j}p(t)  +\sum_{k=1}^{r+1}\left(j+\frac{1}{2}k\right)a_k\delta_{i+j,-k}\omega_0,  \label{even2} \\ \notag \\
[x\otimes t^{i}u,y\otimes t^{j}]&=[x,y]u\otimes t^{i+j}u+ j(x,y)\psi_{ij}(\bf a). \label{odd}
\end{align}

\end{theorem}

\begin{proof}
The identities \eqnref{even1} and \eqnref{even2} follow from \lemref{evenlem} whereas \eqnref{odd} follows from Proposition~\ref{oddlem}. 
\end{proof}
\color{black}

\section{Automorphism group for $R=\mathbb C[t,t^{-1},u: u^2=p(t)=t(t-\alpha_1)\cdots (t-\alpha_{2n})]$.}\label{section:automorphism}

In this section, we restrict to the case of $r=2n$ which allows us to use the results in  \cite{MR3631928}, \cite{MR1223022} and \cite{MR2035219} on automorphism groups of such algebras. 
 \subsection{Automorphisms of $Z(\widehat{\mathfrak{g}})$ of the Current Algebra} 
Let $S_{2n}$ be the symmetry group on the finite set $\{1,2,\ldots, 2n \}$. 

First we recall some background material.   

\begin{theorem}[\cite{MR1223022} and \cite{MR2035219}] The automorphism group of a hyperelliptic curve $A=\mathbb C[X,Y:Y^2=P(X)]$ is isomorphic to one of the following groups:
$$
D_n,\mathbb Z_n,\mathbb V_n,\mathbb H_n,\mathbb G_n,\mathbb U_n,GL_2(3),W_2,W_3
$$
where
\begin{align*}
\mathbb V_n&:=\langle x,y : x^4,y^n,(xy)^2,(x^{-1}y)^2\rangle, \\
\mathbb H_n&:=\langle x,y : x^4,y^2x^2,(xy)^n\rangle ,\\
\mathbb G_n&:=\langle x,y : x^2y^n,y^{2n},x^{-1}yxy\rangle, \\
\mathbb U_n&:=\langle x,y : x^2,y^{2n},xyxy^{n+1} \rangle, \\
W_2&:=\langle x,y : x^4,y^3,yx^2y^{-1}x^2,(xy)^4\rangle, \\
W_3&:=\langle x,y : x^4,y^3,x^2(xy)^4,(xy)^8\rangle.
\end{align*}
\end{theorem}
In \cite{MR2035219} a description of the reduced automorphism group is described for a given polynomial $P(X)$.  In our paper  we don't work with the reduced automorphism group and our coordinate ring is the localization $\mathbb C[t,t^{-1},u: u^2=p(t)=t(t-\alpha_1)\cdots (t-\alpha_{2n})]$ of $A$.

The result below describes the action of automorphisms of the algebra of the hyperelliptic curve $u^2=p(t)$.   The Theorem below corrects some errors that occur in \cite{MR3631928}, Corollary 15.  
 \begin{theorem}[Corollary 15, \cite{MR3631928}]\label{cor:group-actions}  
Let $p(t)=t(t-\alpha_1)\cdots (t-\alpha_{2n})$, where $\alpha_i$ are distinct nonzero roots.   Two possible types of automorphisms $\phi\in \Aut(R_2(p))$ of the algebra $R_2(p)$ are the following: 
 \begin{enumerate}
 \item\label{item:cor-cyclic-thm} If $\alpha_{\gamma(i)}=\zeta \alpha_i$ for some $2n$-th root of unity $\zeta$ and $\gamma\in S_{2n}$, then 
 \begin{equation}\label{eq:cyclic-Z2-case}
 \phi(t)=\zeta t=\xi^2t,\quad \phi(u)=\pm\xi^{2n+1}u=\pm\xi u
 \end{equation}
 where $\xi=\exp(2\pi r\imath/2k)$, $\xi^2=\zeta$ has order $k$ with  $k|2n$ and $r$ and $2k$ are relatively prime. Denote these automorphisms by $\phi_{\xi}^\pm $  which satisfy $(\phi_\xi^\pm)^{2k}=\text{id}$, $(\phi_\xi^+)^k=\phi_1^-$, and $(\phi_\xi^+)^j=\phi_{\xi^j}^+$ for all $j$.
 Consequently  $C_{2k}\cong \langle \phi_\xi^+\rangle$. 
 
 \item\label{cor:group-action-dihedral-maps-thm} If there exists $\gamma\in S_{2n}$ and $c\in\mathbb C$ such that $\alpha_i \alpha_{\gamma(i)}=c^2$ for all $i$, then  $\phi(t)=\zeta t=\xi^2t$ and $\phi(u)=\pm\xi u$ ($\xi$ as above), and 
 $\psi(t)=c^2t^{-1}$ and 
\begin{equation}\label{psia}
 \psi(u)=\pm t^{-n-1}c^{n+1}u\quad \text{ if }a_1=\prod_{i=1}^{2n}\alpha_i=c^{2n},\tag{a}
\end{equation} 
 or 
\begin{equation}\label{psib}
 \psi(u)=\pm t^{-n-1}\imath c^{n+1}u\quad \text{ if }a_1=\prod_{i=1}^{2n}\alpha_i=-c^{2n}.\tag{b}
\end{equation}
Denote these automorphisms by $\psi_c^\pm$, respectively which satisfy  $(\psi_c^\pm)^2=\text{id}$ if $a_1=c^{2n}$ and $(\psi_c^-)^4=\text{id}$ but $(\psi_c^\pm)^2=\phi_1^-$,
if $a_1=-c^{2n}$.

 \end{enumerate}
For case (a) we have if $l=(2n)/k$ is even, then $\text{Aut}(R_2(p))=\langle \phi_\xi^+,\psi_c^+\rangle$ is isomorphic to $D_{2k}=\langle r,s : r^2=s^{2k}=(rs)^2=1\rangle$.
If $l=(2n)/k$ is odd, then  $\text{Aut}(R_2(p))=\langle \phi_\xi^+,\psi_c^+\rangle$ is isomorphic to $\mathbb U_k$. 

For (b) if $l=(2n)/k$ is odd, $\text{Aut}(R_2(p))=\langle \phi_\xi^+,\psi_c^+\rangle$ is isomorphic to $\mathbb V_{2k}$.  If $l$ is even, then $\text{Aut}(R_2(p))=\langle \phi_\xi^+,\psi_c^+\rangle$ is isomorphic to $\mathbb G_k=\text{Dic}_k$. 
 \end{theorem}

 \begin{proof}  Part \eqref{item:cor-cyclic-thm}.
 Let $\phi$ be an automorphism of $R_2(p)$.  Then since the group of units of $R_2(p)$ is $\mathbb C^*  \{t^a : a\in\mathbb Z\}$, we know either $\phi(t)=\zeta t$ for some $\zeta\in\mathbb C^*$ or $\phi(t)=c^2/t$ for some $c\in\mathbb C^*$.  In the first case we have
 \begin{align*}
 \phi(p(t))=\zeta t(\zeta t-\alpha_1)\cdots (\zeta t-\alpha_{2n})=\zeta^{2n+1} t(t-\zeta^{-1}\alpha_1)\cdots (t-\zeta^{-1}\alpha_{2n})=f^2p(t)
 \end{align*}
 as one can show $\phi(u)=fu$ for some $f\in \mathbb C^*\{t^k :  k\in\mathbb Z\}$. 

 Since the $\alpha_i/\zeta$ are distinct we must have that there exists $\gamma\in S_{2n}$ such that $\alpha_{\gamma(i)}=\zeta \alpha_i$ for all $1\leq i\leq 2n$.  Then $\alpha_{\gamma^a(i)}=\zeta\alpha_{\gamma^{a-1}(i)}=\zeta^a\alpha_{i}$.  Suppose $\gamma$ is a product of disjoint cycles $ (c_{i_1},\dots, c_{i_k})$ with $i$ in $\{c_{i_1},\dots ,c_{i_k}\}$.   Then $\alpha_i=\alpha_{\gamma^k(i)}=\zeta^k\alpha_{i}$ for some $k\geq 2$ and $\zeta^k=1$ with $k$ minimal and $\zeta=\exp(2\pi \imath r/k)$ where $r$ is relatively prime to $k$.  Suppose $(d_{i_1},\dots, d_{i_q})$ is an $q$-cycle appearing in $\gamma$. Then $\zeta^l=1$ as well, so $k|q$.  Now 
$$
\alpha_{d_{i_1}},  
\alpha_{d_{i_2}}=\alpha_{\gamma(d_{i_1})}=\zeta \alpha_{d_{i_1}}, 
\dots,  
\alpha_{d_{i_q}}=\zeta^{q}\alpha_{d_{i_1}}
$$ 
are supposed to be distinct. So $\zeta$ is also a $q$-th primitive root of unity, which implies that $q|k$. Thus $q=k$. So
 $\gamma$ is a product of $k$-cycles, and  $k|2n$ 
since $p(t)/t$ has only $2n$ distinct roots.   After reordering the indices we may assume $r=1$. 

In addition  $f^2=\zeta^{2n+1}=\zeta=  \xi^2$ and hence $f=\pm \xi$.   Now $\xi^2=\zeta=\exp(2\pi\imath/k)=(\exp(2\pi\imath/2k))^2$ so that $\xi=\pm \exp(2\pi\imath/2k)$.  We can replace $\xi$ by $-\xi$ in \eqnref{eq:cyclic-Z2-case} if necessary so as to assume $\xi=\exp(2\pi\imath/2k)$. Keep in mind below the fact that $\xi^k=\exp(\pi\imath )=-1$.

It is also easy to check $(\phi_\xi^+)^j=\phi_{\xi^j}^+$.  Let us point out in particular
\begin{equation}
(\phi_\xi^+)^{-1}= \phi_{\xi^{2k-1}}^+=(\phi_\xi^+)^{2k-1}.
\end{equation}

Note also the following
  $$
\phi_{\xi^{k}}^+(t)=(\phi_{\xi}^+)^{k}(t)=\xi^{2k}t=t=\phi_1^-(t),\quad  \phi_{\xi^{k}}^+(u)=(\phi_{\xi}^+)^k(u)=\xi^ku=-u=\phi_1^-(u).
$$
We can thus write
$$
\phi_{\xi^a}^-(t)=\xi^{2a}t=\phi_1^-(\phi_\xi^+)^a(t)=(\phi_\xi^+)^{k+a}(t),\quad \phi_{\xi^a}^-(u)=-\xi^{a}u=\phi_1^-(\phi_\xi^+)^a(u)=(\phi_\xi^+)^{k+a}(u).
$$
Consequently all of the automorphisms of the first type are in the subgroup generated by $\phi_\xi^+$ and this subgroup of $\text{Aut}(R)$ in turn generates a group isomorphic to $C_{2k}$. 

%

We know $kl=2n$ for some positive integer $l$.

Part \eqref{cor:group-action-dihedral-maps-thm}. In the case \eqnref{psia} $c^{2n}=\prod_{i=1}^l\prod_{j=0}^k\zeta^j\alpha_i$,  we have
\begin{align*}
(\psi_c^\pm)^2(t)&=c^2\psi_c^\pm(t^{-1})=t,\quad (\psi_c^\pm)^2(u)=\pm c^{n+1}\psi_c^\pm(t^{-n-1}u)=c^{n+1}c^{-2n-2}t^{n+1}c^{n+1} t^{-n-1}u=u. 
\end{align*}
Then $(\psi_c^\pm)^2=\text{id}$. 

Moreover we have 
\begin{align*}
\psi_c^+\phi_{\xi}^+(\psi_c^+)^{-1}(u) 
&= \psi_c^+\phi_{\xi}^+ \psi_c^+ (u) \\ 
&= \psi_c^+\phi_{\xi}^+ (t^{-n-1} c^{n+1}u )\\ 
&= \psi_c^+(\xi^{-2n-2}t^{-n-1}c^{n+1}\xi u) \\ 
&=(-1)^l \xi^{-1} c^{-2n-2}t^{n+1}c^{n+1}t^{-n-1}c^{n+1}u\\
&= (-1)^l \xi^{-1}u. 
\end{align*}
If $l$ is even (for instance if $k|n$) we conclude
$$
\psi_c^+\phi_{\xi}^+(\psi_c^+)^{-1}(u)=(\phi_{\xi }^{+})^{-1}(u).
$$ 

Furthermore,  
\begin{align*} 
\psi_c^+\phi_{\xi}^+(\psi_c^+)^{-1}(t) & = \psi_c^+\phi_{\xi}^+\psi_c^+(t)  = \psi_c^+\phi_{\xi}^+( c^2t^{-1})\\
&= \psi_c^+(c^2 \xi^{-2}t^{-1}) = c^2\xi^{-2}c^{-2}t = \xi^{-2}t \\
&= (\phi_{\xi}^{+})^{-1}(t)  .
\end{align*} 
Finally note 
$$
\psi_c^-(u)=-c^{n+1}t^{-n-1}u=\phi_1^-\psi_c^+(u),\quad \psi_c^-(t)=c^2/t=\phi_1^-\psi_c^+(t)
$$
so that  for $r=\psi_c^+$,  and $s=\phi_{\xi}^+$ we have $\psi_c^-\in \langle r,s \rangle$. 
In conclusion we have for case (a) with $l$ even,  $r=\psi_c^+$ has order $2$ and $s=\phi_{\xi}^+$ has order $2k$, so they generate the dihedral group
$D_{2k}=\langle r,s : r^2=s^{2k}=(rs)^2=1\rangle$. 

If $l$ is odd, then 
\begin{align*}
 \psi_c^+\phi_{\xi}^+ (\psi_c^+)^{-1} (u) &= -\xi^{-1}u=\phi_\xi^{k-1}(u),\quad   \psi_c^+\phi_{\xi}^+ (\psi_c^+)^{-1} (t)= \psi_c^+\phi_{\xi}^+   (c^{-2}t^{-1}) = \xi^{-2}t=\phi_\xi^{k-1}(t).
\end{align*}
Thus $ \psi_c^+\phi_{\xi}^+ (\psi_c^+)^{-1}= (\phi_{\xi}^+)^{k-1}$ and hence $\psi_c^+\phi_{\xi}^+ \psi_c^+(\phi_\xi^+)^{k+1}=\text{id}$ so  $\langle \phi_\xi^+\rangle$ is a normal subgroup of $\text{Aut}(R_2(p))$ and $\text{Aut}(R_2(p))=\langle \phi_\xi^+,\psi_c^+\rangle\cong \mathbb U_k$.

%
%
In the case \eqnref{psib} $-c^{2n}=\prod_{i=1}^l\prod_{j=1}^k\zeta^j\alpha_i^k$,  we have
\begin{align*}
(\psi_c^\pm)^2(t)&= c^2\psi_c^\pm(t^{-1})=t, \\
 (\psi_c^\pm)^2(u)&=\pm \imath c^{n+1}\psi_c^\pm(t^{-n-1}u)=\imath c^{n+1}c^{-2n-2}t^{n+1}\imath c^{n+1} t^{-n-1}u=-u. 
\end{align*}
Then  $(\psi_c^\pm)^2=\phi_1^-$.

Moreover we have 
\begin{align*}
 \psi_c^+\phi_{\xi}^+ (\psi_c^+)^{-1} (u) &= -\psi_c^+\phi_{\xi}^+ \psi_c^+ (u)  \\
&= -\psi_c^+\phi_{\xi}^+ (t^{-n-1}\imath c^{n+1}u )\\ 
&=- \psi_c^+(\xi^{-2n-2}t^{-n-1}\imath c^{n+1}\xi u) \\ 
&=(-1)^{l+1}\imath \xi^{-1} c^{-2n-2}t^{n+1} c^{n+1}t^{-n-1}\imath c^{n+1}u\\
&= (-1)^l \xi^{-1}u,
\end{align*}
as $\xi^{2n}=\xi^{kl}=(-1)^l$. 
Now if $l$ is odd we conclude
$$
\psi_c^+\phi_{\xi}^+(\psi_c^+)^{-1}(u)=(\phi_{\xi }^{+})^{k-1}(u).
$$ 
In addition 
\begin{align*}
 \psi_c^+\phi_{\xi}^+ (\psi_c^+)^{-1} (t) &= \psi_c^+\phi_{\xi}^+ \psi_c^+ (t)  \\
&= \psi_c^+\phi_{\xi}^+ (c^2/t )\\ 
&= \psi_c^+(c^2\xi^{-2}t^{-1}) \\ 
&=c^2\xi^{-2}c^{-2}t\\
&= \xi^{-2}t=(\phi_\xi^+)^{k-1}(t). 
\end{align*}
Thus
$$
\psi_c^+\phi_{\xi}^+ (\psi_c^+)^{-1} =(\phi_{\xi}^+ )^{k-1}.
$$
Note that this means
$$
\psi_c^+\phi_\xi^+\psi_c^+=\psi_c^+\phi_\xi^+(\psi_c^+)^{-3}=\psi_c^+\phi_\xi^+(\psi_c^+)^{-1}\phi_1^-=(\phi_{\xi}^+ )^{k-1}\phi_1^-=\phi_\xi^{-1}
$$
so that $(\psi_c^+\phi_\xi^+)^2=1$.
Similarly 
 $$
(\psi_c^+)^{-1}\phi_\xi^+(\psi_c^+)^{-1}\phi_\xi^+=(\psi_c^+)^3\phi_\xi^+(\psi_c^+)^{3}\phi_\xi^+=1$$
as $(\psi_c^+)^2=\phi_1^-$.  

 We conclude in the case that $l$ is odd that  
$$
 (\phi_\xi^+)^2=(\phi_\xi^+)^{k},\quad (\phi^+_\xi)^{2k}=\text{id},\quad  (\psi_c^+\phi_{\xi}^+)^2 =1=  ((\psi_c^+)^{-1}\phi_{\xi}^+)^2 $$
and 
$$
\langle \phi_\xi^+,\psi_c^+\rangle \cong \mathbb V_{2k}.
$$

%
When $l$ is even we get 
$$
 (\phi_\xi^+)^2=(\phi_\xi^+)^{k},\quad (\phi^+_\xi)^{2k}=\text{id},\quad  \psi_c^+\phi_{\xi}^+ (\psi_c^+)^{-1} =(\phi_{\xi}^+ )^{-1}
$$
and
$$
\langle \phi_\xi^+,\psi_c^+\rangle \cong \mathbb G_k=\text{Dic}_k.
$$

  \end{proof}
 
 \begin{remark}
In the above cited paper we wrote $(\psi_c^\pm)^2=\text{id}$ but this was in error in case (b) as $\psi_c^\pm$ has order $4$.   Observe also $\phi^-_\xi=\phi_1^-\phi_\xi^+$. 
 \end{remark}

We add to this another 
 \begin{corollary}  \label{polycoefficients}
 Let $p(t)=t(t-\alpha_1)\cdots (t-\alpha_{2n})$, where $\alpha_i$ are distinct roots.   Two possible types of automorphisms $\phi\in \Aut(R_2(p))$ of the algebra $R_2(p)$ are the following: 
 \begin{enumerate}
 \item\label{item:cor-cyclic} If $\alpha_{\gamma(i)}=\zeta \alpha_i$ for some $2n$-th root of unity $\zeta$ and $\gamma\in S_{2n}$, then 
 \begin{equation}
 \phi(t)=\zeta t,\quad \phi(u)=\pm\xi u, 
 \end{equation}
 where we can take $\xi=\zeta^{1/2}=\exp(2\pi\imath/2k)$ with $\zeta$ having order $k$ and  $k|2n$. 
 It follows that $\phi$ has order $2k$. 
 In particular, after a change in indices
\begin{equation}\label{eqn:dihedral-setting-p-in-k-and-n}
 \begin{split}
 p(t) &=
t(t-\alpha_1)(t-\zeta\alpha_1)\cdots (t-\zeta^{k-1}\alpha_1) \cdots (t-\alpha_{2n/k})\cdots (t-\zeta^{k-1}\alpha_{2n/k})\\
&=t(t^k-\alpha_1^k)(t^k-\alpha_2^k)\cdots(t^k-\alpha_{2n/k}^k)    \\
&=\sum_{q=0}^{\frac{2n}{k}}(-1)^qe_q(\alpha_1^k,\dots, \alpha_{2n/k}^k)t^{2n-qk+1}, \\ 
\end{split}
\end{equation}
where $e_q(x_1,x_2,\dots, x_{2n/k})$ is the elementary symmetric polynomial of degree $q$ in $x_1,\dots ,x_{2n/k}$:
\begin{align*}
e_q(x_1,x_2,\dots, x_{2n/k})&=\sum_{1\leq j_1<j_2<\cdots <j_q\leq 2n/k}x_{j_1} x_{j_2} \cdots x_{j_q}.
\end{align*} 
In this case $\langle \phi_\xi^+\rangle\cong C_{2k}$. 
 
 \item\label{cor:group-action-dihedral-maps} If in addition to the above, there exists $\beta\in S_{2n}$ such that $\alpha_i \alpha_{\beta(i)}=c^2$ for all $i$, then  $\phi_\xi^\pm(t)=\zeta t$ and $\phi_\xi^\pm(u)=\pm\xi u$, and 
 $\psi(t)=c^2t^{-1}$ and 
\begin{equation}\label{pma}
 \psi_c^\pm( u)=\pm t^{-n-1} c^{n+1}u \quad \text{ if }a_1=\prod_{i=1}^{2n}\alpha_i=c^{2n},\tag{a}
 \end{equation}
 or 
\begin{equation}\label{pmb}
 \psi_c^\pm(u)=\pm t^{-n-1}\imath c^{n+1}u\quad \text{ if }a_1=\prod_{i=1}^{2n}\alpha_i=-c^{2n}.\tag{b}
\end{equation} 
In this case
\begin{align*}
 p(t)&=\sum_{r=1}^{2n+1}a_rt^r,
\end{align*}
where 
\begin{equation}\label{coefficientsymmetry}
a_k=\pm c^{2n-2k+2}a_{2n+2-k} 
\end{equation}
for $k=1,\dots, 2n+1$.  Here the $\pm$  in \eqnref{coefficientsymmetry} corresponds to the $\pm$ in $a_1=\pm c^{2n}$.

 \end{enumerate}

 \end{corollary}

 \begin{proof}
Case \eqref{item:theorem-cyclic-decomp}.   Thus after a renaming of the indices we may assume $r=1$ and we may write
\begin{align*}
 p(t)&=
t(t-\alpha_1)(t-\zeta\alpha_1)\cdots (t-\zeta^{k-1}\alpha_1)(t-\alpha_2)(t-\zeta\alpha_2)\cdots (t-\zeta^{k-1}\alpha_2)\cdots \\ 
&\hspace{8mm}\cdots (t-\alpha_{2n/k})\cdots (t-\zeta^{k-1}\alpha_{2n/k})\\
&=t(t^k-\alpha_1^k)(t^k-\alpha_2^k)\cdots(t^k-\alpha_{2n/k}^k)    \\
&=\sum_{q=0}^{2n/k}(-1)^qe_q(\alpha_1^k,\dots, \alpha_{2n/k}^k)t^{2n-qk+1}, 
\end{align*}
where $e_q(x_1,\dots, x_{2n/k})$ are the elementary symmetric polynomial of degree $q$ in $x_1,\dots ,x_{2n/k}$
\begin{align*}
e_q(x_1,\dots, x_{2n/k})&=\sum_{1\leq j_1<j_2<\ldots <j_q\leq 2n/k}x_{j_1}\cdots x_{j_q}.
\end{align*}

For the second part we know $C_{2k}\cong \langle \phi_\xi^+\rangle \subseteq \text{Aut}(R_2(p))$ for some $k|2n$ and we have
\begin{align*}
\psi_c^\pm(p(t))=\sum_{j=1}^{2n+1}a_jc^{2j}t^{-j}=t^{-2n-2}\sum_{j=1}^{2n+1}a_jc^{2j}t^{2n+2-j}=t^{-2n-2}\sum_{q=1}^{2n+1}a_{2n+2-q}c^{4n+4-2q}t^{q},
\end{align*}
which we require to satisfy
\begin{align*}
\psi_c^\pm(u^2)&=\psi_c^\pm(p(t))=t^{-2n-2}\sum_{q=1}^{2n+1}a_{2n+2-q}c^{4n+4-2q}t^{q}\\
&=\psi_c^\pm(u)^2=\pm c^{2n+2}t^{-2n-2}p(t)=\pm c^{2n+2}t^{-2n-2}\left(\sum_{j=1}^{2n+1}a_j t^j\right) .\end{align*}
As a consequence (since $c\neq 0$), one has 
$$
a_j=\pm c^{2n-2j+2}a_{2n+2-j},
$$
for $j=1,\dots, 2n+1$.   Here $+$ is taken for case \eqref{pma} and $-$ for case \eqref{pmb}.

\end{proof}

\begin{remark} In the dihedral case with the roots 
$ \zeta^{r}\alpha_i$, where $0\leq r\leq l-1$ and $i=1,\dots , k$ 
so that $kl=2n$, we simplify $|c^{2n}|$ to obtain 
$$
|c^{2n}|=\Big|\prod_{i=1}^{k}\alpha_{i}\Big|^l=|a_1|, 
$$
and thus 
$c^2=\omega \displaystyle{\root n\of {\Big|\prod_{i=1}^{k}\alpha_{i}\Big|^l}}$, where $\omega^n=1$. 

For example if $\alpha_1=1$, $\alpha_2=1+2\imath$, $\alpha_3=1+3\imath$, $l=4$, $n=6$, and $k=3$, then
for any $\gamma\in S_{2n}$,
\begin{align*}
|c^2|=|\alpha_1\alpha_{\gamma(1)}|
&=1\text{ or }\sqrt{5}\text{ or }\sqrt{10}.
\end{align*}
Now for the dihedral group we would also have
$$
 |c^2|=\root n\of {\Big|\prod_{i=1}^{k}\alpha_{i}\Big|^l}= \root 6\of {\sqrt{50}^4}.
$$
But then $ \root 6\of {\sqrt{50}^4}\neq |\alpha_1\alpha_{\gamma(1)}|$ for any $\gamma$.  As a consequence, the automorphism group is $C_{6}$, and not $D_6$. 
\end{remark}
%
%
%


From \cite{MR966871}, we know that for any
automorphism $\phi$ of the associative algebra $R_2(p)$, one obtains
an automorphism $\tau$ of the Lie algebra $\mathcal{R}_2(p):=\text{Der}(R_2(p))$ through the equation
\begin{equation}\label{eq:Liealgauto}
\tau(f(t)\partial)=\phi(f)(\phi\circ\partial\circ \phi^{-1})\enspace \mbox{ for all } f \in R_2(p).
\end{equation}
In addition, any Lie algebra automorphism of
$\mathcal{R}_2(p)$ can be obtained from \eqnref{eq:Liealgauto}. Denote by
$\tau_{\zeta}^{\pm}$ and $\sigma_{c}^{\pm}$ the Lie algebra
automorphisms corresponding to the associative algebra automorphisms
$\phi_{\zeta}^{\pm}$ and $\psi_{c}^{\pm}$ in \thmref{cor:group-actions}
\eqref{item:cor-cyclic-thm} and \eqref{cor:group-action-dihedral-maps-thm} respectively (if they indeed exist). For convenience,
denote
$$
\tau_{\zeta}=\tau_{\zeta}^{+} \mbox{ and } \sigma_{c}=\sigma_{c}^{+}.$$

Let $C_k$ be the cyclic group of order $k$ and $D_k$ be the dihedral group of order $2k$.  
 \begin{corollary}[\cite{MR3631928}, Corollary 16]\label{cor:automorphism-groups} 
Let $p(t)=t(t-\alpha_1)\cdots (t-\alpha_{2n})$, with distinct roots. 
 \begin{enumerate}
\item  If $\sigma_c^{\pm}$ does not exist in $\Aut(\mathcal{R}_2(p))$ for any nonzero complex number $c$, then 
$\Aut(\mathcal{R}_2(p))$ is generated by the automorphism $\tau_{\zeta}^+$ of order $2k$, where $k|2n$. In otherwords we have 
\[ \Aut(\mathcal{R}_2(p))= \langle \tau_{\zeta}^+ \rangle \simeq C_{2k}. 
\]  
 \item  If $\sigma_c^{\pm}$ exists in $\Aut(\mathcal{R}_2(p))$ for some nonzero complex number $c$ with $c^{2n}=a_1$, then $\Aut(\mathcal{R}_2(p))$ is generated by  $\sigma_c^+$, and some automorphism $\tau_{\zeta}^+$ of order $2k$, where $k|2n$. If $k|n$, then we have 
 \[ 
 \Aut(\mathcal{R}_2(p)) = 
 \langle \tau_{\zeta}^+, \sigma_c^+ \rangle \simeq D_{2 k} .
 \]  
   \end{enumerate}
 \end{corollary}

\begin{proof}
This follows from Theorem~\ref{cor:group-actions}. 

\end{proof}


\section{The decomposition of the space of K\"ahler differentials modulo exact forms for $D_{2k}$}\label{section:main-theorem}
Let $r=2n$, $R=R_2(p)$ and 
let $G:=\Aut(R)$ be the groups in Corollary~\ref{cor:automorphism-groups}. 
For $\phi\in G$ and $\overline{rds} \in \Omega_R/dR$,  
the action of $G$ on the K\"ahler differential is given by: 
\begin{equation}
\phi(\overline{rds}) = \overline{\phi(r)d\phi(s)}. \label{groupaction}
\end{equation}

First we note the following:

\begin{lemma}
For $n$ even, the character table is given by the matrix
\footnotesize
$$
M=\bordermatrix{\text{}&\rho_1&\rho_2&\rho_3 & \rho_4 & \chi_1 & \ldots &\chi_h & \ldots & \chi_{(n/2)-1} \cr
              1 &1 &  1  & 1 & 1 & 2 & \ldots & 2& \ldots & 2 \cr
               \psi & 1  &  -1& 1 & -1 & 0 &\ldots & 0 & \ldots & 0 \cr
                \psi \phi& 1 & -1 & -1& 1 & 0 &\ldots & 0 & \ldots & 0 \cr
                \phi & 1  &   1      &-1 & -1  & 2 \cos\left( \frac{2\pi}{ n}\right) &\ldots & 2\cos\left(\frac{2\pi h}{n}\right) & \ldots & 2\cos\left(\frac{2\pi \left((n/2)-1\right)}{n}\right)    \cr 
                \vdots & \vdots &  \vdots      &\vdots & \vdots  & \vdots & \ddots & \vdots & \ddots & \vdots  \cr 
                \phi^k& 1  &   1      &(-1)^k & (-1)^k   & 2 \cos\left( \frac{2\pi k}{ n}\right) &\ldots & 2\cos\left(\frac{2\pi hk}{n} \right) & \ldots & 2\cos\left(\frac{2\pi k \left((n/2)-1\right)}{n}\right)    \cr 
                \vdots  & \vdots  &   \vdots        &\vdots  & \vdots & \vdots & \ddots & \vdots & \ddots & \vdots \cr 
                \phi^{n/2} & 1  &   1      &(-1)^{n/2} & (-1)^{n/2}   & -2 &\ldots & 2(-1)^h& \ldots & 2(-1)^{(n/2)-1} \cr 
}.$$
So we have 
\begin{align*}
M^{-1} &= M^t \begin{pmatrix}
       \frac{1}{2n} &0 &  0  & 0 &    \ldots & 0& \ldots & \ldots & 0 \\
       0 &  \frac{n/2}{2n}& 0 & 0 &\ldots  & 0 & \ldots &  \ldots &0\\
       0& 0 &  \frac{n/2}{2n} & 0 &\ldots & 0 & \ldots &  \ldots & 0 \\
       0 & 0 & 0 &\frac{2}{2n} &\ldots & 0 & \ldots &  \ldots &0    \\ 
       \vdots & \vdots &\vdots & \vdots  & \ddots & \vdots  & \ddots & \vdots & \vdots  \cr 
       0 & 0 & 0 & 0 &\ldots &\frac{2}{2n}& \ldots & \vdots & 0  \\
       \vdots & \vdots &\vdots  & \vdots & \ddots & \vdots & \ddots & \vdots & \vdots \\
       0& 0 & 0 & 0 & \ldots & \vdots & \ldots &  \frac{2}{2n} & 0    \\
       0  & 0 & 0 & 0 & \ldots & 0& \ldots &  0& \frac{1}{2n} \\ 
\end{pmatrix}\\
&=\begin{pmatrix}
               \frac{1}{2n}  & \frac{1}{4} &  \frac{1}{4}  & \frac{1}{n} &    \ldots & \frac{1}{n}& \ldots &  \frac{1}{2n} \\
                \frac{1}{2n}   &  -\frac{1}{4}&- \frac{1}{4} &  \frac{1}{n} &  \ldots  & \frac{1}{n} & \ldots &  \frac{1}{2n} \\
                   \frac{1}{2n}  & \frac{1}{4} & -\frac{1}{4}& - \frac{1}{n}  &\ldots &  \frac{(-1)^k }{n}& \ldots &   \frac{(-1)^{n/2}}{2n}   \\
                \frac{1}{2n}  &   -\frac{1}{4}      & \frac{1}{4} & - \frac{1}{n} & \ldots  & \frac{(-1)^k}{n} & \ldots &\frac{(-1)^{n/2} }{2n}   \\ 
                \frac{1}{n}   &   0 & 0 &  \frac{2}{n} \cos\left( 2\pi /n\right)  &\ldots &  \frac{2}{n}\cos\left(2\pi h/n\right) & \ldots & -\frac{1}{n}   \\ 
                  \vdots &  \vdots      &\vdots & \vdots  & \ddots & \vdots & \ddots & \vdots\\
                   \frac{1}{n}   & 0 &0      & \frac{2}{ n}\cos\left(2\pi h/n\right)  &\ldots &  \frac{2}{n}\cos\left(2\pi hk /n\right) & \ldots & \frac{(-1)^h }{n} \\
                 \vdots  &   \vdots        &\vdots  & \vdots & \ddots &\vdots & \ddots & \vdots \\
                 \frac{1}{n}   &   0      & 0 &  \frac{2}{n}\cos\left(2\pi  \left((n/2)-1\right)/n\right)   &\ldots & \frac{2}{n}\cos\left(2\pi k \left((n/2)-1\right)/n\right)  &\ldots   & \frac{(-1)^{(n/2)-1}}{n} \\ 
\end{pmatrix}.  \\
\end{align*}

\end{lemma}
\begin{proof}  Observe $M$ is just the character table for $D_n$ when $n$ is even:  
from page 37 in \cite{MR0450380}, we obtain the character table for $r=2n$ 
\[ 
\vline 
\begin{array}{c|c|c}
\hline
 & (\phi_{\zeta}^+)^k   &  \psi_c^+ (\phi_{\zeta}^+)^k  \\  \hline 
\rho_1 & 1 & 1   \\    \hline 
\rho_2 &  1 &  -1 \\    \hline 
\rho_3 &  (-1)^k &  (-1)^k  \\    \hline 
\rho_4 &  (-1)^k &  (-1)^{k+1}  \\    \hline 
\chi_h &  2\cos\left(2\pi hk/n\right)  &  0  \\    \hline 
\end{array} 
\vline ,
\]  
where $1\leq h<n/2$ for $n$ even, $\psi_c^+$ is a reflection, and $\phi_{\zeta}^+$ is a rotation.

Let $\Xi$ denote the set of conjugacy classes of the group $D_n$.  Then from the orthogonality of the characters of the irreducible representations, we get the inverse matrix for $M$ since one needs the following formula for any two irreducible representations $\pi$ and $\rho$ of $D_n$: 
\begin{equation}\label{orthongalityofchar}
\sum_{\{g\}\in\Xi }\frac{|\{g\}|}{|D_n|}\chi_\pi(g)\overline{\chi_\rho(g)}=\begin{cases} 1 
& \text{ for }\pi\cong \rho, \\
0 & \text{ otherwise}, \end{cases}
\end{equation}
(see page 260 in \cite{MR1695775}).
 
   The distinct conjugacy classes of $D_n$ via conjugation (for $n=2\widehat{m}$ even) are: 
 \[ 
 \begin{aligned}
 \{\Inoindex \},  
 &\{ \phi_{\zeta}^+, (\phi_{\zeta}^+)^{-1}\},  \ldots, 
  \{ (\phi_{\zeta}^+)^j, (\phi_{\zeta}^+)^{-j}\}, \ldots, 
    \{ (\phi_{\zeta}^+)^{\frac{n}{2}-1}, (\phi_{\zeta}^+)^{-(\frac{n}{2}-1)}\}, 
  \{ (\phi_{\zeta}^+)^{\widehat{m}}\}, \\ 
  &\{\psi_c^+ (\phi_{\zeta}^+)^{\ell}: \ell \mbox{ even}, 
  0\leq \ell \leq n-2 \}, 
  \{\psi_c^+ (\phi_{\zeta}^+)^{p}: p \mbox{ odd},  1 \leq p\leq n-1 \} 
\end{aligned} 
 \]  
since the even dihedral group has nontrivial center (thus giving us one element orbits). 
\end{proof}

So under an action by $G$, 
we decompose $\Omega_R/dR = Z(\widehat{\mathfrak{g}})$ 
into a direct sum of irreducible representations.  
Our goal in this section is to describe the module structure of $\Omega_R/dR$ into irreducibles under the action by $G$ for a particular $R_2(p)$. Recall that $\mathcal{R}_2(p) = \text{Der}(R_2(p))$.

\begin{theorem}\label{theorem-result-001}  
Let 
$p(t)=t(t-\alpha_1)\cdots (t-\alpha_{2n})$, where $\alpha_i$ are pairwisely distinct nonzero complex numbers.
\begin{enumerate}
\item\label{item:theorem-cyclic-decomp} 
 If $\sigma_c^{\pm}$ does not exist in $\Aut(\mathcal{R}_2(p))$ for any nonzero $c\in \mathbb{C}$,  then $\Aut(\mathcal{R}_2(p))\cong C_{2k}$ 
 and the center for the universal central extension of $\mathfrak g\otimes R_2(p)$ decomposes as: 
\begin{equation}
\Omega_R/dR\cong U_0\oplus \ldots \oplus U_{k-1}, 
 \end{equation}  
 where $\displaystyle{U_r=\bigoplus_{i\equiv r\mod k,1\leq i\leq 2n}\mathbb C\omega_i}$ for $r=1,\ldots, k-1$ is a sum of one-dimensional irreducible representation of $C_{2k}$ with character $\chi_r(s)=\exp(2\pi\imath rs/2k)$, each occurring with multiplicity $l$ and  
 $$
 U_0=\mathbb C\omega_0\oplus \bigoplus_{i=1}^l \mathbb C\omega_{ki}. 
 $$
\item\label{item:theorem-dihedral-decomp} 
Assume $\sigma_c^{\pm}$ exists in $\Aut(\mathcal{R}_2(p))$ for some nonzero $c\in \mathbb{C}$, $c^{2n}=a_1$ and $k|n$.  
Then $\Aut(\mathcal{R}_2(p)) \cong D_{2k}$. 
Moreover, if $k$ is also even then under the action of $D_{2k}$ the center decomposes as: 
\begin{equation} 
\Omega_R/dR \cong \mathbb{C} \omega_0 \oplus  \bigoplus_{i=3}^{4} U_i^{\frac{(1-(-1)^k)n}{2k} } \oplus \bigoplus_{h=1}^{k-1} V_{h}^{\oplus \frac{(1-(-1)^h)n}{k}}. 
\end{equation} 
where $U_i$, $i=1,2,3,4$, are the irreducible one dimensional representations for $D_{2k}$ with character $\rho_i$ and $V_h$ are the irreducible 2-dimensional representations for $D_{2k}$ with character $\chi_h$, $1\leq h\leq k-1$.  Note $\mathbb C\omega_0$ and $U_1$ are the trivial representations.

When $k$ is odd, the center decomposes as
\begin{equation}
\Omega_R/dR \cong \mathbb{C}\omega_0 
\oplus \bigoplus_{i=3}^{4} U_i^{\oplus \Upsilon_i(\epsilon_i, \nu_i)} 
\oplus 
\bigoplus_{j=1}^{k-1} V_j^{\oplus \frac{(1-(-1)^j)n}{k}}
\end{equation}
with 
\begin{align*} 
\Upsilon_i(\epsilon_i,\nu_i) 
  &= \frac{(1-(-1)^k)n}{2k}(\delta_{i,3}+\delta_{i,4})+\frac{1-(-1)^n}{4}(\delta_{i,4}-\delta_{i,3})+\frac{1 }{2} (-1)^i  \displaystyle{\sum_{i=n+3}^{2n}}
c^{n+3-2i}
  P_{i-n-3,-i}.
\end{align*} 
 \end{enumerate}
 \end{theorem}


\begin{corollary}\label{cor:honest-2-dimensional-irreps}
When $\overline{\omega_i}=c^{-\frac{n+3-2i}{2}}\zeta^{-\frac{i}{2}}\omega_i$ for $1\leq i\leq n+2$, we obtain that 
\begin{center}
$\overline{\omega_i}$ and 
$\overline{\omega_{n+3-i}}$, where $1\leq i\leq (n+2)/2$, 
\end{center}
span a $2$-dimensional irreducible representation for $n$ even.
 \end{corollary}

The following is a proof of Theorem~\ref{theorem-result-001}. 
 
\begin{proof} 
\vskip 10pt
We will first prove \eqref{item:theorem-cyclic-decomp}.  Recalling \eqref{eq:cyclic-Z2-case}:   
\[ 
 \phi(t)=\xi^2t=\zeta t \mbox{ and } \phi(u)=\xi  u  
 \] 
(so $\phi=\phi_\xi^+$ has order $2k$), the action of $\Aut_{\mathbb{C}}(R)$ shows that 
\[ 
\phi^j(\omega_0) =\phi^j(\overline{t^{-1}dt})
=\overline{\phi^j(t)^{-1}\,d\phi^j(t)}
=\overline{\zeta^{-j}t^{-1}\,d(\zeta^j t)}=\overline{t^{-1}\,d t}=\omega_0 
\]  
and 
\begin{equation}
\begin{aligned}
\phi^j(\omega_i)&
=\phi^j(\overline{t^{-i}udt})
=\overline{\xi^{-2ij}t^{-i}\xi^j u\, d(\xi^{2j} t) }
= \overline{ \xi^{(3-i)j}t^{-i}u\, dt}
=\xi^{(3-i)j}  \omega_i\\
&=\exp(2\pi\imath(3-2i)j/2k)\omega_i=\exp(2\pi\imath(kl-2i+3)j/2k)\omega_i
\end{aligned}
\end{equation}
for all $0\leq j\leq 2k$ and $0\leq i\leq 2n$.  
Now the characters of the irreducible representations of $C_{2k}$ are of the form $\chi_h(\phi^s)=\exp\left(2\pi\imath sh/2k\right)$ with $0\leq h\leq 2k-1$.
In order to figure out the multiplicities, we need to solve the number of solutions to 
$$
2s\equiv 2r\mod 2k
$$
for $1\leq r\leq s\leq 2n$.  In this case $2(s-r)=2kd$, so $0\leq s-r=kd\leq 2n-1=kl-1$ for some integer $d$.  Thus $s=r+kd$ where $0\leq d\leq l-1$ and the multiplicity is $l$ for each irreducible representation.
 
We conclude that the center $\Omega_R/dR$ decomposes into the direct sum of one-dimensional eigenspaces:  
\begin{equation} 
\Omega_R/dR\cong U_0\oplus \ldots \oplus U_{k-1},  
 \end{equation}   
 where  
 $$ 
 \displaystyle{U_r=\bigoplus_{i\equiv r\mod k,1\leq i\leq 2n}\mathbb C\omega_i}
\hspace{2mm} \mbox{ for } r=1,\ldots, k-1,  
$$  
a sum of one-dimensional irreducible representation of $C_{2k}$ with character $\chi_r(s)=\exp(2\pi\imath rs/2k)$, each occurring with multiplicity  $l$  and 
 $$ 
 U_0=\mathbb C\omega_0\oplus \bigoplus_{i=1}^l \mathbb C\omega_{ki}. 
 $$ 
where $U_i$ are the one dimensional irreducible representations of $D_{2k}$ with characters $\rho_i$, $i=1,2,3,4$ and $V_h$ are the irreducible representations with character $\chi_h$, $1\leq h\leq k-1$.

Next, we see that 
\[ 
\begin{aligned}
\psi_c^{+} (\omega_0) &= \psi_c^{+}(\overline{t^{-1}\, dt}) 
= \overline{c^{-2}t \, d(c^2t^{-1}) } 
= \overline{t\, d(t^{-1}) } 
=- \overline{t\cdot t^{-2}dt}  
= -\omega_0  
\end{aligned}
\] 
and 
\[ 
\phi_{\xi}^+(\omega_0) = \phi_{\xi}^+(\overline{t^{-1}dt})
= \overline{\xi^{-2}t^{-1}d(\zeta t)} 
= \overline{t^{-1}dt} = \omega_0. 
\] 
So $\omega_0$ is a basis element for a one-dimensional irreducible representation under the action of $D_{2k}$. 

Similarly, we have the rotations acting on $\omega_i$ as a scalar multiplication: 
\begin{align}\label{eqn:phi-zeta-plus}
\phi_{\xi}^+(\omega_i) &= \phi_{\xi}^+( \overline{ t^{-i}udt } ) 
= \overline{\zeta^{-i}t^{-i}\xi u\zeta dt}    
= \overline{\xi^{3-2i}t^{-i}u dt }   \\
&= \xi^{3-2i} \omega_i\notag
\end{align} 
 and the reflections acting via: 
\begin{align}\label{eqn:psi-c-plus} 
\psi_c^+(\omega_i) &= \psi_c^+(\overline{t^{-i}udt}) 
= \overline{c^{-2i} t^i t^{-n-1}c^{n+1}ud(c^2t^{-1}) }  
= - \overline{ c^{n-2i+3}t^{i-n-1}ut^{-2} dt }   \\ 
&= - \overline{c^{n+3-2i}t^{-(n+3-i)} udt } \notag \\
&= -c^{n+3-2i} \omega_{n+3-i}\quad\text{ if }1\leq i\leq n+2  \notag
\end{align}  
where we assumed $a_1=c^n$. 

%
We also have: 
\[ 
\begin{aligned} 
\psi_c^+(\omega_{n+3-i}) &= \psi_c^+(\overline{ t^{-n-3+i}u\, dt} ) 
= \overline{c^{-2n-6+2i}t^{n+3-i}t^{-n-1} c^{n+1}u\, d(c^2t^{-1}) } \\ 
&
= - \overline{ c^{-n-3+2i}t^{-i} u\, dt } \\ 
&= -c^{-(n+3-2i)} \omega_i,\quad \text{ if } 1\leq i\leq n+2
\end{aligned} 
\] 
i.e., $\sigma_c^+(-c^{n+3-2i} \omega_{n+3-i})= \omega_i$.

\underline{Case 1}. Let $n$ be even (but different from $2$). 
We see that for $1\leq i\leq \frac{n+2}{2}$, the $2$-dimensional spaces $\mathbb{C}\omega_i\oplus \mathbb{C}\omega_{n+3-i}$ form irreducible $D_{2k}$-representations
since  the matrix representation for $\phi_{\zeta}^+$ and $\psi_c^+$ with respect to the basis $\omega_i$ and $\omega_{n+3-i}$ are: 
\begin{equation}
\phi_{\zeta}^+|_{\{\omega_i,\omega_{n+3-i}\}} = 
\left( 
\begin{array}{cc}
\zeta^{\frac{2n+3-2i}{2}} & 0 \\ 
0 & \zeta^{\frac{2i-3}{2}}  \\ 
\end{array}
\right)  
\mbox{ and } 
\psi_c^+ |_{\{\omega_i,\omega_{n+3-i}\}}
= 
\left( 
\begin{array}{cc} 
0&  -c^{-(n+3-2i)} \\ 
-c^{n+3-2i} & 0 \\ 
\end{array}
\right), 
\end{equation}
respectively, where $\tr(\phi_{\zeta}^+|_{\{\omega_i,\omega_{n+3-i}\}}) = \zeta^{n} $ and 
$\tr(\psi_c^+ |_{\{\omega_i,\omega_{n+3-i}\}})=0$. 
 It follows from Corollary~\ref{cor:honest-2-dimensional-irreps} that we indeed have $2$-dimensional irreducible representations.  

For $i$ between $n+3 \leq i \leq 2n$,  
\[ 
\overline{t^{i-n-3}u\,dt} = \sum_{k=1}^{2n}P_{i-n-3,-k}\omega_k
\] 
by Equation~\eqref{eqn:omega_i-Pij-Qij}.  
Thus for $n+3\leq i\leq 2n$, we have 
\begin{equation}\label{eqn:psi-c-plus-last} 
\psi_{c}^+(\omega_i) =-c^{n+3-2i} \overline{t^{i-n-3}u\,dt} = -c^{n+3-2i}  \sum_{k=1}^{2n}P_{i-n-3,-k} \omega_k. 
\end{equation}
Recall the recursion relations: 
\begin{equation}
(2k+r+3)P_{l,-i}=-\sum_{j= 1 }^{r} (3j+2l-2r)a_jP_{l-r+j-1,-i}
\end{equation}
for $l\geq 0$ with the initial condition
$P_{-m,-i}=\delta_{-m,-i}$, $1 \leq i,m\leq r$. Now from \corref{polycoefficients} we have $a_j=0$ unless $j=1+qk$ for some $0\leq q\leq  (2n)/k$. Hence we have 
\begin{align*}
(2k+2n+3)P_{l,-i}&=-\sum_{j= 1 }^{2n} (3j+2l-4n)a_jP_{l-2n+j-1,-i} \\
&=-\sum_{q= 0 }^{\frac{2n}{k}-1} (3qk+2l-4n+3)a_{1+qk}P_{l-2n+qk,-i}
\end{align*}
so for a summand on the right to be nonzero we must have $l-2n+qk=-i+ak$ for some $a\in\mathbb Z$. Or rather $l=-i+2n+(a-q)k=-i+bk$ for some $b\in\mathbb Z$.
Otherwise it might be that $P_{l,-i}$ is be nonzero for $l=-i+bk$. 

  In particular if $l=i-n-3$, then $l\neq -i+bk$  for any $b\in\mathbb Z$ (otherwise $l\equiv i-3\mod k$ and $l\equiv -i\mod k$ gives us $i-3=-i\mod k$ and $2i=3+dk$ with $k$ even).  Hence 
$P_{i-n-3,-i}=0$.

The matrix representation for $\eqref{eqn:phi-zeta-plus}$ in basis $\{ \omega_1,\ldots, \omega_{2n}\}$ is 
\[ 
\phi_{\xi}^+ = 
\left(
\begin{array}{ccccc}
\xi & 0& 0& 0& 0  \\ 
0 & \xi^{-1}& 0& 0& 0  \\ 
0 & 0& \xi^{-3}& 0& 0  \\ 
\vdots & \vdots& \vdots& \ddots & \vdots  \\ 
0 & 0& 0& 0&  \xi^{3-2n} \\ 
\end{array}
 \right), 
\] 
which is traceless,
while the matrix representation for \eqref{eqn:psi-c-plus} for $1\leq i\leq n+2$ and \eqref{eqn:psi-c-plus-last} for $n+3\leq i\leq 2n$ in  $\{ \omega_1,\ldots, \omega_{2n}\}$ is  
\[ 
\psi_c^+ = 
\left( 
\begin{array}{ccccccccc}
0 &0 & \cdots & 0& -c^{-(n+1)}&-c^{-(n+3)} P_{0,-1}& \ldots & -c^{3-3n} P_{n-3,-1}\\ 
0 &0 & \cdots & -c^{-(n-1)} &0 & -c^{-(n+3)}P_{0,-2}& \ldots & -c^{3-3n}P_{n-3,-2} \\ 
\vdots  &\vdots & .^{.^.} &\vdots  &\vdots  & \vdots &.^{.^.} & \vdots\\ 
0& -c^{n-1} & \ldots & 0& 0& -c^{-(n+3)}P_{0,-n-1} &\ldots & -c^{3-3n}P_{n-3,-n-1} \\ 
-c^{n+1} &0 &\ldots &0 &0 &-c^{-(n+3)}P_{0,-n-2}& \ldots& -c^{3-3n}P_{n-3,-n-2}\\ 
0& 0& \vdots & 0& 0& -c^{-(n+3)}P_{0,-n-3}&\vdots   &  -c^{3-3n}P_{n-3,-n-3}\\ 
\vdots &\vdots &.^{.^.} &\vdots &\vdots &  \vdots & .^{.^.}& \vdots\\  
0 &0 &\ldots &0 &0 &  -c^{-(n+3)}P_{0,-2n+2}& \ldots & -c^{3-3n}P_{n-3,-2n+2} \\ 
0 &0 &\ldots  &0 &0 &-c^{-(n+3)}P_{0,-2n+1}  &\ldots  &   -c^{3-3n}P_{n-3,-2n+1}\\ 
0 &0 &\ldots  &0 &0 &  -c^{-(n+3)} P_{0,-2n} & \ldots &-c^{3-3n}P_{n-3,-2n}  \\ 
\end{array}
\right),  
\] 
which has trace   
$$
-\sum_{i=n+3}^{2n} c^{n+3-2i}
  P_{i-n-3,-i}   = 0 
  $$  
  since $k|n$ and $k$ is even. 

So for $n$ even, the set of equations we need to solve is 
  $$
  \chi_{(\Omega_R/dR)/\mathbb C\omega_0}=n_1\rho_1+n_2\rho_2+n_3\rho_3+n_4\rho_4+\sum_{h=1}^{k-1}m_h\chi_h, 
  $$
  which are precisely, 
  \begin{align*}
  2n&=n_1 +n_2 +n_3 +n_4 +\sum_{h=1}^{k-1}2m_h, \\
    0 
  &=n_1 -n_2 +n_3 -n_4 \quad\text{ for }\psi_c^+, \\
  0 
&=n_1-n_2 -n_3 +n_4 \quad\text{ for }\psi_c^+\phi_\xi^+, \\
 0&=n_1 +n_2 +(-1)^qn_3 +(-1)^qn_4 +\sum_{h=1}^{k-1}2m_h\cos (2\pi hq/2k)\quad \mbox{ for } 1\leq q\leq (2k/2)-1=k-1 \\
 -2n&=n_1 +n_2 +(-1)^qn_3 +(-1)^qn_4 +\sum_{h=1}^{k-1}2m_h(-1)^h.
\end{align*}
 In the above we used the fact that for $1\leq j\leq k$ one has
 $$
  \chi_{(\Omega_R/dR)/\mathbb C\omega_0}((\phi_\xi^+)^j)= \sum_{i=1}^{2n}\xi^{(3-2i)j}=\xi^{3j}\sum_{i=1}^{2n}\xi^{-2ij}=2n\xi^{3j}\delta_{j,k}=-2n\delta_{j,k}
 $$
 since
\begin{align*}
 0&=\xi^{2nj}-1=(\xi^{2j}-1)(\xi^{2j(n-1)}+\xi^{2j(n-2)}+\cdots +\xi^{2j}+1)  \\
 &=(\xi^{2j}-1)(\xi^{-2j}+\xi^{-4j}+\cdots +\xi^{2j}+1)=(\xi^{2j}-1)\left(\sum_{i=1}^{2n}\xi^{-2ij}\right).
 \end{align*}
 If $1\leq j<k$, then the left factor in the last equality is not zero so the sum must be zero.
The set of equations above can be written as 
\begin{align*}
M\begin{pmatrix}n _1 \\
n_2 \\ 
n_3 \\
n_4 \\
m_1\\
\vdots \\
m_h\\
\vdots \\
m_{n/2-1}\end{pmatrix}=\begin{pmatrix}
2n  \\
  0 \\
0  \\ 
  0 \\  
0 \\ 
\vdots  \\ 
0 \\ 
0\\
-2n\\
\end{pmatrix}.
\end{align*}
 Thus 
 \begin{align*}
&\begin{pmatrix}n _1 \\
n_2 \\ 
n_3 \\
n_4 \\
m_1\\
\vdots \\
m_h\\
\vdots \\
m_{n/2-1}\end{pmatrix}
=M^{-1}\begin{pmatrix}
2n  \\
 0  \\
0     \\
  0 \\ 
  0 \\
\vdots  \\
0\\
0\\ 
-2n\\ 
\end{pmatrix}  \\
&= 
\begin{pmatrix}
   \frac{1}{4k}  & \frac{1}{4} &  \frac{1}{4}  & \frac{1}{2k} &    \ldots & \frac{1}{2k}& \ldots &  \frac{1}{4k} \\
   \frac{1}{4k}   &  -\frac{1}{4}&- \frac{1}{4} &  \frac{1}{2k} &  \ldots  & \frac{1}{2k} & \ldots &  \frac{1}{4k} \\
   \frac{1}{4k}  & \frac{1}{4} & -\frac{1}{4}& - \frac{1}{2k}  &\ldots &  \frac{(-1)^{\ell} }{2k}& \ldots &   \frac{(-1)^{ k  }}{4k}   \\
   \frac{1}{4k}  & -\frac{1}{4} & \frac{1}{4} & - \frac{1}{2k} & \ldots  & \frac{(-1)^{\ell}}{2k} & \ldots &\frac{(-1)^{k} }{4k}   \\ 
   \frac{1}{2k}   & 0 & 0 &  \frac{2}{2k} \cos\left( \frac{2\pi}{2k} \right)  &\ldots &  \frac{2}{2k}\cos\left( \frac{2\pi h}{2k}  \right) & \ldots & -\frac{1}{2k}   \\ 
   \vdots &  \vdots &\vdots & \vdots  & \ddots &\vdots   & \ddots & \vdots\\
   \frac{1}{2k}   & 0 &0 & \frac{2}{2k}\cos\left(  \frac{2\pi h}{2k}  \right)  &\ldots &  \frac{2}{2k}\cos\left(  \frac{2\pi h\ell}{2k}   \right) & \ldots & \frac{(-1)^h }{2k} \\
   \vdots  &   \vdots &\vdots  & \vdots & \ddots & \vdots & \ddots & \vdots \\
   \frac{1}{2k} & 0 & 0 &  \frac{2}{2k}\cos\left( \frac{2\pi  \left(\frac{2k}{2}-1\right)}{2k}\right)  
 & \ldots & \frac{2}{2k}\cos\left(\frac{2\pi \ell \left(\frac{2k}{2}-1\right)}{2k}\right)  &\ldots   & -\frac{(-1)^{k}}{2k} \\ 
\end{pmatrix}
\begin{pmatrix}
2n  \\
 0     \\
 0   \\
 0 \\ 
 0 \\ 
 0 \\ 
\vdots  \\
0 \\ 
-2n \
\end{pmatrix}    
=\begin{pmatrix}
0 \\
0  \\
\frac{(1-(-1)^k)n}{2k}   \\
\frac{(1-(-1)^k)n}{2k}  \\
\frac{2n}{k}   \\ 
0 \\ 
\vdots  \\ 
\frac{(1-(-1)^h)n}{k} \\ 
\vdots \\
\frac{(1+(-1)^k)n}{k} 
\end{pmatrix}.   
\end{align*}
\normalsize

In the case where $a_1=c^{2n}$, $l=(2n)/k$ is even but $k$ is odd, the multiplicities of irreducible representations are given by 
 \begin{align*}
&\begin{pmatrix}n _1 \\
n_2 \\ 
n_3 \\
n_4 \\
m_1\\
\vdots \\
m_h\\
\vdots \\
m_{n/2-1}\end{pmatrix}
=M^{-1}\begin{pmatrix}
2n  \\
-\frac{1-(-1)^n}{2}-\displaystyle{\sum_{i=n+3}^{2n}}
c^{n+3-2i}
  P_{i-n-3,-i}  \\
\frac{1-(-1)^n}{2}-\displaystyle{\sum_{i=n+3}^{2n}}
\xi^{3-2i}
c^{n+3-2i}
  P_{i-n-3,-i}  \\ 
  0 \\ 
  \vdots \\
  0 \\
-2n 
\end{pmatrix}  \\
&= 
\tiny 
\begin{pmatrix}
   \frac{1}{4k}  & \frac{1}{4} &  \frac{1}{4}  & \frac{1}{2k} &    \ldots & \frac{1}{2k}& \ldots &  \frac{1}{4k} \\
   \frac{1}{4k}   &  -\frac{1}{4}&- \frac{1}{4} &  \frac{1}{2k} &  \ldots  & \frac{1}{2k} & \ldots &  \frac{1}{4k} \\
   \frac{1}{4k}  & \frac{1}{4} & -\frac{1}{4}& - \frac{1}{2k}  &\ldots &  \frac{(-1)^{\ell} }{2k}& \ldots &   \frac{(-1)^{ k  }}{4k}   \\
   \frac{1}{4k}  & -\frac{1}{4} & \frac{1}{4} & - \frac{1}{2k} & \ldots  & \frac{(-1)^{\ell}}{2k} & \ldots &\frac{(-1)^{k} }{4k}   \\ 
   \frac{1}{2k}   & 0 & 0 &  \frac{2}{2k} \cos\left( \frac{2\pi}{2k} \right)  &\ldots &  \frac{2}{2k}\cos\left( \frac{2\pi h}{2k}  \right) & \ldots & -\frac{1}{2k}   \\ 
   \vdots &  \vdots &\vdots & \vdots  & \ddots &\vdots   & \ddots & \vdots\\
   \frac{1}{2k}   & 0 &0 & \frac{2}{2k}\cos\left(  \frac{2\pi h}{2k}  \right)  &\ldots &  \frac{2}{2k}\cos\left(  \frac{2\pi h\ell}{2k}   \right) & \ldots & \frac{(-1)^h }{2k} \\
   \vdots  &   \vdots &\vdots  & \vdots & \ddots & \vdots & \ddots & \vdots \\
   \frac{1}{2k} & 0 & 0 &  \frac{2}{2k}\cos\left( \frac{2\pi  \left(\frac{2k}{2}-1\right)}{2k}\right)  
 & \ldots & \frac{2}{2k}\cos\left(\frac{2\pi \ell \left(\frac{2k}{2}-1\right)}{2k}\right)  &\ldots   & \frac{(-1)^{k+1}}{2k} \\ 
\end{pmatrix}
\normalsize 
\begin{pmatrix}
2n  \\
-\frac{1-(-1)^n}{2}-\displaystyle{\sum_{i=n+3}^{2n}}
c^{n+3-2i}
  P_{i-n-3,-i}  \\
\frac{1-(-1)^n}{2}-\displaystyle{\sum_{i=n+3}^{2n}}
\xi^{3-2i}
c^{n+3-2i}
  P_{i-n-3,-i}  \\ 
  0 \\ 
\vdots  \\
0\\
-2n
\end{pmatrix} \\ 
&=
\normalsize 
\begin{pmatrix}
  - \frac{1}{4}   \displaystyle{\sum_{i=n+3}^{2n}}
c^{n+3-2i}
  P_{i-n-3,-i} - \frac{1}{4}\displaystyle{\sum_{i=n+3}^{2n}}
\xi^{3-2i}
c^{n+3-2i}
  P_{i-n-3,-i} \\
  \frac{1}{4} \displaystyle{\sum_{i=n+3}^{2n}}
c^{n+3-2i}
  P_{i-n-3,-i} + \frac{1}{4}\displaystyle{\sum_{i=n+3}^{2n}}
\xi^{3-2i}
c^{n+3-2i}
  P_{i-n-3,-i} \\
\frac{(1-(-1)^k)n}{2k}-\frac{1-(-1)^n}{4} - \frac{1}{4} \displaystyle{\sum_{i=n+3}^{2n}}
c^{n+3-2i}
  P_{i-n-3,-i} + \frac{1}{4}\displaystyle{\sum_{i=n+3}^{2n}}
\xi^{3-2i}
c^{n+3-2i}
  P_{i-n-3,-i}  \\
\frac{(1-(-1)^k)n}{2k}+\frac{1-(-1)^n}{4}+ \frac{1}{4} \displaystyle{\sum_{i=n+3}^{2n}}
c^{n+3-2i}
  P_{i-n-3,-i} - \frac{1}{4}\displaystyle{\sum_{i=n+3}^{2n}}
\xi^{3-2i}
c^{n+3-2i}
  P_{i-n-3,-i} \\
\frac{2n}{k}\\ 
\vdots  \\ 
\frac{(1-(-1)^h)n}{k} \\ 
\vdots \\
\frac{(1+(-1)^k)n}{k} 
\end{pmatrix}.   
\end{align*}
\normalsize 
Observe now that 
\begin{align*}
0\leq n_1&=   - \frac{1}{4}   \displaystyle{\sum_{i=n+3}^{2n}}
c^{n+3-2i}
  P_{i-n-3,-i} - \frac{1}{4}\displaystyle{\sum_{i=n+3}^{2n}}
\xi^{3-2i}
c^{n+3-2i}
  P_{i-n-3,-i}\\
  0 \leq n_2&=  \frac{1}{4} \displaystyle{\sum_{i=n+3}^{2n}}
c^{n+3-2i}
  P_{i-n-3,-i} + \frac{1}{4}\displaystyle{\sum_{i=n+3}^{2n}}
\xi^{3-2i}
c^{n+3-2i}
  P_{i-n-3,-i}
\end{align*}
so that 
$$
   \displaystyle{\sum_{i=n+3}^{2n}}
c^{n+3-2i}
  P_{i-n-3,-i} =- \displaystyle{\sum_{i=n+3}^{2n}}
\xi^{3-2i}
c^{n+3-2i}
  P_{i-n-3,-i}.
$$

\end{proof}

 We will now prove Corollary~\ref{cor:honest-2-dimensional-irreps}.

 \begin{proof}
Let $n$ be even. 
We change the basis to   
\[ \overline{\omega_i}= c^{-\frac{n+3-2i}{2}}\zeta^{-\frac{i}{2}}\omega_i\mbox{ for } 
1\leq i\leq n+2  
\] 
to obtain that we indeed have $2$-dimensional irreducible representations. 
Since 
\[ 
\overline{\omega_{n+3-i}} 
= c^{\frac{n+3-2i}{2}}\zeta^{-\frac{n+3-i}{2}} 
\omega_{n+3-i} \mbox{ for } 1\leq i\leq n+2, 
\] 
we have 
\begin{enumerate}[label=(\roman*)]
\item 
$\phi_{\zeta}^+(\overline{\omega_i}) = c^{-\frac{n+3-2i}{2}}\zeta^{-\frac{2i}{4}}\zeta^{\frac{2n+3-2i}{2}}\omega_i 
= \zeta^{\frac{2n+3-2i}{2}}\overline{\omega_i}$, 
\item 
$\phi_{\zeta}^+(\overline{\omega_{n+3-i}}) 
= c^{\frac{n+3-2i}{2}}\zeta^{-\frac{n+3-i}{2}}\zeta^{\frac{-3+2i}{4}}\omega_{n+3-i}
= \zeta^{-\frac{2n+3-2i}{2}}\overline{\omega_{n+3-i}}
$, 
\item $\psi_{c}^+(\overline{\omega_i}) 
=-c^{-\frac{n+3-2i}{2}}\zeta^{-\frac{2i}{4}} c^{n+3-2i}\omega_{n+3-i}
=\zeta^{\frac{2n+3-2i}{2}}\overline{\omega_{n+3-i}}$, 
\item $\psi_{c}^+(\overline{\omega_{n+3-i}}) =-
c^{\frac{n+3-2i}{2}}\zeta^{-\frac{n+3-i}{2}}
c^{-\frac{n+3-2i}{2}}\omega_{i} 
= \zeta^{-\frac{2n+3-2i}{2}}\overline{\omega_i}
$. 
\end{enumerate}
With respect to the basis $\{ \overline{\omega_1},\ldots, \overline{\omega_{n+2}}\}$, 
this implies 
\[ 
\phi_{\zeta}^+\Big|_{\{ \overline{\omega_i},\overline{\omega_{n+3-i}}
\}} = 
\begin{pmatrix}
\zeta^{\frac{2n+3-2i}{2}} & 0 \\ 
0 & \zeta^{-\frac{2n+3-2i}{2}} \\ 
\end{pmatrix}
\] 
and 
\[
\psi_c^+ \Big|_{\{ \overline{\omega_i},\overline{\omega_{n+3-i}}
\}} = 
\begin{pmatrix}
0 & \zeta^{-\frac{2n+3-2i}{2}}\\ 
\zeta^{\frac{2n+3-2i}{2}} & 0 \\ 
\end{pmatrix}, 
\] 
which coincide with classical $2$-dimensional irreducible representations for dihedral groups.

Now, let $n$ be odd. 
With respect to the basis 
\[ 
\overline{\omega_i} = c^{-\frac{n+3-2i}{4}}\zeta^{-\frac{i}{2}}\omega_i \mbox{ for } 1\leq i\leq n+2, 
\] 
we have 
\[ 
\phi_{\zeta}^+\Big|_{\{ \overline{\omega_i},\overline{\omega_{n+3-i}}
\}} = 
\begin{pmatrix}
\zeta^{\frac{2n+3-2i}{2}} & 0 \\ 
0 & \zeta^{-\frac{2n+3-2i}{2}} \\ 
\end{pmatrix}
\] 
and 
\[
\psi_c^+ \Big|_{\{ \overline{\omega_i},\overline{\omega_{n+3-i}}
\}} = 
\begin{pmatrix}
0 & \zeta^{-\frac{2n+3-2i}{2}}\\ 
\zeta^{\frac{2n+3-2i}{2}} & 0 \\ 
\end{pmatrix}, 
\]  
and we note that 
 \begin{equation}
\psi_c^+(\overline{\omega_{\frac{n+3}{2}}})   
= -\overline{\omega_{\frac{n+3}{2}}}  
\hspace{4mm} 
	\mbox{ and } 
\hspace{4mm}
 \phi_{\zeta}^+(\overline{\omega_{\frac{n+3}{2}}}) 
 = \zeta^{n/2} \overline{\omega_{\frac{n+3}{2}}}
 = - \overline{\omega_{\frac{n+3}{2}}}.  
 \end{equation}

 \end{proof}

\begin{example}
In the case when $n=3$ and $k=3$ for $p(t)=t (t^3 -\alpha_1^3) (t^3 - \alpha_2^3)$,

\begin{align*}
\psi_{c}^{+}  = 
\begin{pmatrix}
0 & 0& 0& 0& -\frac{1}{c^2} & 0 \\ 
0& 0& 0& -\frac{1}{c} & 0  & 0 \\ 
0& 0& -1 & 0 & 0 & 0 \\ 
0& -c& 0 & 0 & 0 & 0 \\ 
-c^2& 0& 0 & 0 & 0 & 0 \\ 
0& 0& 0 & 0 & 0 & -\frac{\alpha_1^3\alpha_2^3}{c^6} \\ 
\end{pmatrix}.
\end{align*}
In this case $\displaystyle{\alpha_1^3\alpha_2^3=c^6}$, the trace of $\psi_{c}^{+}$ equals $-2$, giving us multiplicities 
\begin{align*}
n_1 = n_2=n_3=0, n_4 = 2, m_1 = 2. 
\end{align*}
\end{example}

\begin{example}
For $n=6$ and $k=3$, we have 
\begin{align*}
p(t)&=t(t^3-\alpha_1^3)(t^3-\alpha_2^3)(t^3-\alpha_3^3)(t^3-\alpha_4^3)  \\
&=t^{13}-\left(\alpha _1^3+\alpha _2^3+\alpha _3^3+\alpha _4^3\right) t^{10}
+ \left(
      \alpha _1^3 \alpha _2^3
   + \alpha _1^3 \alpha _3^3 
   + \alpha _1^3 \alpha _4^3       
   + \alpha _2^3 \alpha _3^3
   + \alpha _2^3 \alpha _4^3
   + \alpha _3^3 \alpha _4^3 
   \right) t^7 \\
   &\quad -\left(
   \alpha _1^3 \alpha _2^3 \alpha _3^3
   + \alpha _1^3 \alpha _2^3 \alpha _4^3   
   + \alpha_1^3 \alpha _3^3 \alpha _4^3 
   + \alpha _2^3 \alpha _3^3 \alpha _4^3 
   \right)   t^4
   +\alpha _1^3 \alpha _2^3 \alpha_3^3 \alpha _4^3 t
\end{align*}
Then $\psi_c^+$ is 
\[ 
\left(
\begin{array}{cccccccccccc}
 0 & 0 & 0 & 0 & 0 & 0 & 0 & -\frac{1}{c^{7}} & 0 & 0 & \Lambda_3 & 0 \\
 0 & 0 & 0 & 0 & 0 & 0 & -\frac{1}{c^{5}} & 0 & 0 & \Lambda_2 & 0 & 0 \\
 0 & 0 & 0 & 0 & 0 & -\frac{1}{c^{3}} & 0 & 0 & \Lambda_1 & 0 & 0 & \Theta_4 \\
 0 & 0 & 0 & 0 & -\frac{1}{c} & 0 & 0 & 0 & 0 & 0 & \Theta_3 & 0 \\
 0 & 0 & 0 & -c & 0 & 0 & 0 & 0 & 0 & \Theta_2 & 0 & 0 \\
 0 & 0 & -c^{3} & 0 & 0 & 0 & 0 & 0 & \Theta_1 & 0 & 0 &\Delta_4\\
 0 & -c^{5} & 0 & 0 & 0 & 0 & 0 & 0 & 0 & 0 & \Delta_3 & 0 \\
 -c^{7} & 0 & 0 & 0 & 0 & 0 & 0 & 0 & 0 &\Delta_2 & 0 & 0 \\
 0 & 0 & 0 & 0 & 0 & 0 & 0 & 0 & \Delta_1 & 0 & 0 & \Gamma_4\\
 0 & 0 & 0 & 0 & 0 & 0 & 0 & 0 & 0 & 0 & \Gamma_3 & 0 \\
 0 & 0 & 0 & 0 & 0 & 0 & 0 & 0 & 0 & \Gamma_2 & 0 & 0 \\
 0 & 0 & 0 & 0 & 0 & 0 & 0 & 0 & \Gamma_1 & 0 & 0 & \Psi \\
\end{array}
\right),
\] 
where 
\begin{align*}
\Lambda_1 &= -\frac{2 \left(\alpha _1^3+\alpha
   _2^3+\alpha _3^3+\alpha _4^3\right)}{5 c^{9}}=\frac{2 a_{10}}{5 c^{9}}, \\ 
\Lambda_2 &= -\frac{8 \left(\alpha _1^3+\alpha
   _2^3+\alpha _3^3+\alpha _4^3\right)}{17 c^{11}}= \frac{8 a_{10}}{17 c^{11}}, \\ 
\Lambda_3 &= -\frac{10 \left(\alpha _1^3+\alpha
   _2^3+\alpha _3^3+\alpha _4^3\right)}{19 c^{13}}= \frac{10a_{10}}{19 c^{13}}, \\ 
\Theta_1 &= -\frac{   
     \alpha _1^3  \alpha _2^3 + \alpha _1^3  \alpha _3^3 + \alpha _1^3  \alpha_4^3  
   +   \alpha _2^3 \alpha _3^3+  \alpha _2^3  \alpha_4^3   
   + \alpha _3^3 \alpha _4^3 
   }{5 c^{9}}= -\frac{   a_7
   }{5 c^{9}}, \\ 
\Theta_2 &=  -\frac{ 
  \alpha _1^3  \alpha _2^3 + \alpha _1^3  \alpha _3^3+\alpha _1^3  \alpha_4^3 
    +  \alpha _2^3 \alpha _3^3+\alpha _2^3 \alpha_4^3 
       + \alpha _3^3 \alpha _4^3
   }{17 c^{11}}=-\frac{ a_7}{17 c^{11}}, \\
\Theta_3 &= \frac{
\alpha _1^3  \alpha _2^3 + \alpha _1^3  \alpha _3^3+\alpha _1^3  \alpha_4^3 
    +  \alpha _2^3 \alpha _3^3+\alpha _2^3 \alpha_4^3 
       + \alpha _3^3 \alpha _4^3    
   }{19 c^{13}}= \frac{a_7}{19 c^{13}},  \\ 
   \Theta_4 &= -\frac{
   8 (\alpha _1^6+  \alpha _2^6+ \alpha_3^6+ \alpha _4^6) 
   +
  11  \left( \alpha _1^3\alpha _2^3 + \alpha _1^3\alpha _3^3 + \alpha _1^3 \alpha _4^3  
 + \alpha _2^3 \alpha _3^3 + \alpha _2^3 \alpha_4^3
  + \alpha _3^3 \alpha _4^3 \right) 
   }{
   35 c^{15}
   } \\
   &= -\frac{
   8 (a_{10}^2-2a_7) +11 a_7 }{35 c^{15} }= -\frac{
   8 a_{10}^2-5 a_7 }{35 c^{15} }, \\ 
   \Delta_1 &=  \frac{4 \left(  \alpha _1^3\alpha _2^3\alpha _3^3+ \alpha _1^3\alpha _2^3\alpha _4^3+ \alpha _1^3\alpha _3^3 \alpha _4^3+\alpha _2^3 \alpha _3^3 \alpha
   _4^3\right)}{5 c^{9}}= - \frac{4a_4}{5 c^{9}},  \\ 
   \Delta_2 &=  \frac{
   10 \left( \alpha _1^3  \alpha _2^3  \alpha _3^3+ \alpha _1^3  \alpha _2^3 \alpha_4^3 
   +\alpha _1^3   \alpha _3^3 \alpha _4^3  
   + \alpha _2^3 \alpha_3^3 \alpha _4^3\right) }{17 c^{11}}= - \frac{10a_4}{17 c^{11}}, \\ 
   \Delta_3 &= \frac{8 \left(   \alpha _1^3 \alpha _2^3  \alpha _3^3 +  \alpha _1^3 \alpha _2^3  \alpha_4^3 
   +  \alpha _1^3 \alpha _3^3 \alpha _4^3  + \alpha _2^3 \alpha
   _3^3 \alpha _4^3\right)}{19 c^{13}}=- \frac{8 a_4}{19 c^{13}}, \\ 
   \Delta_4 &=  
   -\frac{2}{35 c^{15}} \left( \alpha _1^3 \alpha _2^3 \alpha _3^3
    + \alpha _1^3 \alpha _2^3 \alpha _4^3   
    + \alpha _1^3 \alpha _3^3 \alpha _4^3 
    + \alpha _2^3 \alpha _3^3 \alpha _4^3  \right.\\
  &\hskip 10pt\left.
   +2
   \left(\alpha _1^6 (\alpha _2^3 +\alpha _3^3  +\alpha _4^3  )  
   +\alpha _2^6 (\alpha_1^3+ \alpha _3^3 +\alpha _4^3 )
   +\alpha _3^6 (\alpha_1^3+ \alpha _2^3  + \alpha_4^3)
   +\alpha _4^6 (\alpha_1^3+ \alpha _2^3  +\alpha _3^3  )     
   \right) \right) \\ 
   &=  
   -\frac{2}{35 c^{15}} \left( 5a_4-2a_7a_{10}   \right), \\
   \Gamma_1 &= -\frac{7 \alpha _1^3 \alpha _2^3 \alpha _3^3 \alpha _4^3}{5
   c^{9}}, \\ 
   \Gamma_2 &=  -\frac{19 \alpha _1^3 \alpha _2^3 \alpha _3^3 \alpha
   _4^3}{17 c^{11}}, \\ 
   \Gamma_3 &=  -\frac{17 \alpha _1^3 \alpha _2^3 \alpha _3^3
   \alpha _4^3}{19 c^{13}}, \\ 
   \Gamma_4 &=  
   \frac{1}{35 c^{15}} 
   \left(
16 
\left( 
 \alpha _1^6 ( \alpha _2^3 \alpha _3^3 +\alpha _2^3 \alpha _4^3  +  \alpha_3^3 \alpha _4^3   )      
   +  \alpha _2^6 ( \alpha _1^3 \alpha_3^3  +\alpha _1^3  \alpha _4^3+\alpha _3^3  \alpha _4^3  )
    \right. \right.\\
  &\hskip 10pt\left. \left. 
  +  \alpha _3^6 (\alpha _1^3 \alpha _2^3 + \alpha _1^3 \alpha _4^3 +  \alpha _2^3  \alpha _4^3 )
 + \alpha _4^6  (\alpha _1^3   \alpha _2^3  +  \alpha _1^3 \alpha _3^3  +  \alpha _2^3 \alpha_3^3    )   
   \right) 
   + 39  \alpha _1^3  \alpha _2^3 \alpha _3^3 \alpha _4^3   
   \right)  \\ 
   &=  
   \frac{1}{35 c^{15}} 
   \left(16 a_4 a_{10}-25 a_1\right),  \\ 
   \Psi &= -\frac{4 \alpha _1^3 \alpha _2^3 \alpha _3^3 \alpha _4^3 \left(\alpha
   _1^3+\alpha _2^3+\alpha _3^3+\alpha _4^3\right)}{5 c^{15}}= -\frac{4   \left(\alpha
   _1^3+\alpha _2^3+\alpha _3^3+\alpha _4^3\right)}{5 c^{3}}= \frac{4   a_{10}}{5 c^{3}}. \\ 
\end{align*}
By \eqnref{coefficientsymmetry} we have
$a_{10}=c^{-6}a_4$
so that 
\begin{align*}
\tr \psi_c^+ &= \Delta_1+\Psi = -\frac{4 a_4}{5 c^{9}}  +\frac{4 a_{10}}{5 c^{3}} =0.   \end{align*}
   This implies that the multiplicities appearing are 
   \[ 
   n_1 = n_2 = 0,\quad n_3 =   n_4 = 2 \mbox{ and } m_1 =4,\quad m_2 = 0. 
   \]

\end{example}
 
\begin{example} When $n=9$ and $k=3$, we used Mathematica to get 
$$
\tr(\psi_c^+)=-2=-\tr(\psi_c^+\phi_\xi^+), 
$$
and hence
$$
n_1=n_2=0,n_3=2,n_4=4,m_1=6,m_2=0.
$$

\end{example}


\def\cprime{$'$} \def\cprime{$'$} \def\cprime{$'$} \def\cprime{$'$}
\providecommand{\bysame}{\leavevmode\hbox to3em{\hrulefill}\thinspace}
\providecommand{\MR}{\relax\ifhmode\unskip\space\fi MR }
\providecommand{\MRhref}[2]{%
  \href{http://www.ams.org/mathscinet-getitem?mr=#1}{#2}
}
\providecommand{\href}[2]{#2}

\appendix

\end{document}